\documentclass[a4paper,12pt]{amsart}
\usepackage{amsmath,amssymb,amsthm,amscd}
\usepackage{enumerate,bm,color,here,url,hyperref}
\usepackage{thmtools}

\usepackage{graphicx}
\graphicspath{{./figure/}}

\makeatletter
\@namedef{subjclassname@2020}{\textup{2020} Mathematics Subject Classification}
\makeatother

\newtheoremstyle{roman}
{3pt}
{3pt}
{\rmfamily}
{}
{\rmfamily}
{.}
{0.5em}
{}

\theoremstyle{plain}
	\newtheorem{thm}{Theorem}[section]
	\newtheorem{lem}[thm]{Lemma}
	\newtheorem{prop}[thm]{Proposition}
	\newtheorem{cor}[thm]{Corollary}

	\newtheorem{ques}[thm]{Question}
	
\theoremstyle{roman}
	\newtheorem{dfn}[thm]{Definition}
	\newtheorem{rem}[thm]{Remark}

\newcommand{\fig}[3][width=12cm]{
\begin{figure}[htbp]
	\centering 
	\includegraphics[#1,clip]{#2} 
	\caption{#3} 
\label{fig:#2}
\end{figure}}

\DeclareMathOperator{\dev}{dev}

\DeclareMathOperator{\Fix}{Fix}
\DeclareMathOperator{\hol}{hol}
\DeclareMathOperator{\Hom}{Hom}
\DeclareMathOperator{\Homeo}{Homeo}
\DeclareMathOperator{\Imm}{Imm}
\DeclareMathOperator{\inter}{int}
\DeclareMathOperator{\Isom}{Isom}

\DeclareMathOperator{\Mon}{Mon}
\DeclareMathOperator{\Out}{Out}
\DeclareMathOperator{\Stab}{Stab}
\DeclareMathOperator{\tr}{tr}
\DeclareMathOperator{\Tw}{Tw}
\DeclareMathOperator{\vol}{vol}

\newcommand{\bbZ}{\mathbb{Z}}
\newcommand{\bbR}{\mathbb{R}}
\newcommand{\bbC}{\mathbb{C}}
\newcommand{\bbH}{\mathbb{H}}
\newcommand{\calHC}{\mathcal{HC}}
\newcommand{\calC}{\mathcal{C}}
\newcommand{\calX}{\mathcal{X}}

\begin{document}

\title{Holed cone structures on 3-manifolds}
\author{Ken'ichi YOSHIDA}
\address{International Institute for Sustainability with Knotted Chiral Meta Matter (WPI-SKCM$^2$), Hiroshima University, 1-3-1 Kagamiyama, Higashi-Hiroshima, Hiroshima 739-8531, Japan}
\email{kncysd@hiroshima-u.ac.jp}

\subjclass[2020]{57M50 (Primary), 57K31, 57K32 (Secondary)}
\keywords{Cone-manifolds, Holonomy representations, Character varieties, Volumes of representations}
\date{}

\begin{abstract}
We introduce holed cone structures on 3-manifolds 
to generalize cone structures. 
In the same way as a cone structure, 
a holed cone structure induces the holonomy representation. 
We consider the deformation space consisting of the holed cone structures on a 3-manifold 
whose holonomy representations are irreducible. 
This deformation space for positive cone angles 
is a covering space on a reasonable subspace of the character variety. 
\end{abstract}

\maketitle

\section{Introduction}
\label{section:intro}

A cone-manifold is a generalization of a Riemannian manifold of constant sectional curvature, 
allowed to have cone singularity. 
The rotational angle around cone singularity is not equal to $2\pi$, 
and it is called the cone angle. 
While a finite volume complete hyperbolic structure on a 3-manifold admits no continuous deformation 
by the Mostow rigidity, 
it can often be deformed via hyperbolic cone-manifolds. 
For example, a hyperbolic Dehn surgery (in the strong sense) 
gives continuous deformation from a cusped hyperbolic 3-manifold to a hyperbolic 3-manifold obtained by gluing solid tori 
via hyperbolic cone-manifolds whose singular locus consists of the cores of glued solid tori. 
If the cone angles of a cone-manifold are equal to $2\pi/n_{i}$ for $n_{i} \in \bbZ_{> 0}$, 
this cone-manifold can be regarded as an orbifold. 
Proofs of the geometrization of 3-orbifolds in \cite{BLP05, CHK00} are based on this fact. 

Local and global rigidities of hyperbolic cone-manifolds are known as follows. 
For an oriented 3-manifold $X$ and an $n$-component link $\Sigma$ in $X$, 
let $\calC_{[0,\theta]} (X,\Sigma)$ denote 
the space of hyperbolic cone structures on $(X,\Sigma)$ with cone angles at most $\theta$. 
Let $\Theta \colon \calC_{[0,\theta]} (X,\Sigma) \to [0,\theta]^{n}$ denote 
the map assigning the cone angles. 
The local rigidity by Hodgson and Kerckhoff~\cite{HK98} 
states that $\Theta \colon \calC_{[0,2\pi]} (X,\Sigma) \to [0,2\pi]^{n}$ 
is a local homeomorphism. 
The global rigidity by Kojima~\cite{Kojima98} 
states that $\Theta \colon \calC_{[0,\pi]} (X,\Sigma) \to [0,\pi]^{n}$ 
is injective. 
The global rigidity is not known if some cone angles exceed $\pi$. 
Izmestiev~\cite{Izmestiev11} gave examples 
where the global rigidity does not hold, and the cone angles exceed $2\pi$. 
The proof in \cite{Kojima98} is based on the fact 
that two cone loci with cone angles less than $\pi$ are not close. 
A cone structure in $\calC_{[0,\pi]} (X,\Sigma)$ 
can be continuously deformed to the cusped hyperbolic structure on $X \setminus \Sigma$. 
However, this argument does not work if cone angles exceed $\pi$. 
Cone structures may degenerate by meeting cone loci 
even if the cone angles decrease \cite{Yoshida22}. 

In this paper, 
we introduce the notion of holed cone structures 
to generalize cone structures. 
A holed cone structure on $(X,\Sigma)$ is defined as 
an equivalence class of cone metrics outside some balls in $X$. 
In the same way as a non-holed one, a holed cone structure induces 
the holonomy representation of $\pi_{1} (X \setminus \Sigma)$ to $\Isom^{+} (\bbH^{3})$ 
up to conjugation. 
We consider the deformation space $\calHC^{\mathrm{irr}} (X,\Sigma)$ 
consisting of the holed cone structures on $(X,\Sigma)$ 
whose holonomy representations are irreducible. 
In Theorem~\ref{thm:hausdorff}, we will show that 
$\calHC^{\mathrm{irr}} (X,\Sigma)$ is Hausdorff. 
We will define a subspace $\calX^{\mathrm{cone}}(X,\Sigma)$ 
of the character variety $\calX (\pi_{1} (X \setminus \Sigma))$. 
The $\widehat{\calX}^{\mathrm{cone}}(X,\Sigma)$ consists of 
the pairs of elements in $\calX^{\mathrm{cone}}(X,\Sigma)$ and compatible cone angles. 
The map 
\[
\widehat{\hol} \colon \calHC^{\mathrm{irr}}(X,\Sigma) \to \widehat{\calX}^{\mathrm{cone}}(X,\Sigma)
\] 
assigns the holonomy representation and the cone angles. 
In Theorem~\ref{thm:cover}, we will show that 
the map 
\[
\widehat{\hol} \colon \calHC^{\mathrm{irr}}_{+}(X,\Sigma) \to \widehat{\calX}^{\mathrm{cone}}_{+}(X,\Sigma)
\] 
is a regular covering map to each path-connected component of 
$\widehat{\calX}^{\mathrm{cone}}_{+}(X,\Sigma)$ 
that contains an image, 
where the map is restricted to the elements with positive cone angles. 
Lemma~\ref{lem:subgroup} implies that 
the map 
\[
\widehat{\hol} \colon \calHC^{\mathrm{irr}}(X,\Sigma) \to \widehat{\calX}^{\mathrm{cone}}(X,\Sigma)
\] 
is not injective. 
Consequently, global rigidity for holed cone structures does not hold 
even if $X \setminus \Sigma$ admits a hyperbolic structure. 

Section~\ref{section:cone} concerns 
(non-holed) cone structures in the space of holed cone structures. 
In Theorem~\ref{thm:conehol}, we will show that 
cone metrics $g$ and $g^{\prime}$ with $\widehat{\hol} (g) = \widehat{\hol} (g^{\prime})$ 
are equivalent. 
By Corollary~\ref{cor:coneeq}, 
we may regard a cone structure as a holed cone structure. 
We expect that the notion of holed cone structures is useful 
to consider global rigidity for cone-manifolds. 
It should be worthwhile to ask 
whether a cone structure can be deformed to the cusped hyperbolic structure 
via holed cone structures. 

In Section~\ref{section:vol}, 
we will introduce the volume of a holed cone structure. 
This is defined as 
the sum of the volume of a holed cone metric and the volume enclosed by the holes. 
After all, 
this volume is equal to the volume of the holonomy representation 
by Theorem~\ref{thm:vol}. 

In Section~\ref{section:ex}, 
we will give an explicit example of holed cone structures. 
This is an extension of the construction of cone structures by the author~\cite{Yoshida22}. 
This example illustrates how holed cone structures enable us 
to avoid degeneration by meeting cone loci. 

Euclidean and spherical holed cone structures can be defined in the same manner. 
To show the corresponding results, 
it is necessary to take care of the topologies of quotients of representation spaces.

\section{Definition of holed cone structures}
\label{section:def}

Hyperbolic metrics on a manifold 
are Riemannian metrics with constant sectional curvature $-1$. 
Equivalently, 
they are modeled by the hyperbolic space $\bbH^{n}$ of constant sectional curvature $-1$. 
According to \cite{CHK00}, 
a (hyperbolic) cone-manifold 
is a topological manifold 
with a complete path metric (called a \emph{cone metric}) 
which can be triangulated 
into hyperbolic simplices. 
Although it does not matter in 3 dimensions, the link of each simplex in this triangulation 
is needed to be piecewise linearly homeomorphic to a standard sphere. 

The \emph{singular locus} of a cone-manifold 
consists of the points with no neighborhood isometric to a hyperbolic ball. 
We consider 3-dimensional cone-manifolds 
whose singular locus consists of simple closed geodesics. 
Then locally a cone metric has the form 
\[
dr^{2} + \sinh^{2} r d \theta^{2} + \cosh^{2} r dz^{2} 
\]
in cylindrical coordinates around an axis, 
where $r$ is the distance from the axis, 
$z$ is the distance along the axis, 
and $\theta$ is the angle measured modulo the \emph{cone angle} $\theta_{0} > 0$. 
If a cone angle is equal to $2\pi$, then the metric is smooth around the point. 

Let $X$ be an oriented 3-manifold, 
and let $\Sigma$ be a union of disjoint circles in $X$. 
A cone metric on $(X,\Sigma)$ is a metric on $X$ 
such that $X$ is a cone-manifold with singular locus $\Sigma$. 
We allow that a cone angle is equal to $0$ or $2\pi$. 
By generalizing the notion, 
we say that a component $\Sigma_{i}$ of $\Sigma$ has cone angle zero 
if the metric around $\Sigma_{i}$ is a cusp neighborhood of $\Sigma_{i}$. 
In this case, the metric is defined outside $\Sigma_{i}$. 
We call a component of $\Sigma$ a \emph{cone locus}. 
A \emph{cone structure} on $(X,\Sigma)$ is an equivalence class 
of a cone metric on $(X,\Sigma)$, 
where two metrics are \emph{equivalent} 
if there is an isometry isotopic to the identity between them, 
and the intermediate maps in the isotopy preserve $\Sigma$ setwise. 

We introduce holed cone structures on $(X,\Sigma)$ 
as a generalization of cone structures. 

\begin{dfn}
Let $B$ be a union of finitely many (possibly zero) disjoint closed 3-balls 
in $X \setminus \Sigma$. 
A \emph{holed cone metric} on $(X,\Sigma)$ 
is a cone metric $g$ on $(X \setminus \inter (B), \Sigma)$ 
with smooth boundary $\partial B$. 
We call each component of $B$ a \emph{hole}. 
We call the metric space $(X \setminus \inter (B), \Sigma; g)$ a \emph{holed cone-manifold}. 
\end{dfn}

\begin{dfn}
\label{dfn:equivalence}
Let $g$ and $g^{\prime}$ be holed cone metrics on $(X,\Sigma)$ 
respectively with holes $B$ and $B^{\prime}$. 
The metrics $g$ and $g^{\prime}$ are \emph{equivalent} 
if there are holed cone metrics $g_{i}$ with holes $B_{i}$ on $(X,\Sigma)$ for $0 \leq i \leq n$ 
such that 
$g_{0} = g$, $g_{n} = g^{\prime}$, and for each $0 \leq i \leq n-1$ either 
\begin{enumerate}
\item there is a map $f \colon (X,\Sigma) \to (X,\Sigma)$ isotopic to the identity 
	(the intermediate maps in the isotopy preserve $\Sigma$ setwise) 
	such that the restriction of $f$ to $(X \setminus \inter (B_{i}), \Sigma; g_{i})$ 
	is an isometry onto $(X \setminus \inter (B_{i+1}), \Sigma; g_{i+1})$,  
\item $B_{i} \subset B_{i+1}$, and 
	$g_{i+1}$ is the restriction of $g_{i}$ to 
	$X \setminus \inter (B_{i+1})$, or 
\item $B_{i+1} \subset B_{i}$, and 
	$g_{i}$ is the restriction of $g_{i+1}$ to 
	$X \setminus \inter (B_{i})$. 
\end{enumerate}
We call an equivalence class $[g]$ a \emph{holed cone structure}. 
The relation (1) will not be mentioned explicitly. 
\end{dfn}

A cone metric is a holed cone metric by definition. 
Moreover, we can say that a cone structure is a holed cone structure. 
We will prove it in Corollary~\ref{cor:coneeq}. 

If we remove an embedded ball with smooth boundary 
disjoint from the cone loci in a holed cone-manifold, 
then we obtain an equivalent holed cone metric. 
Conversely, 
if a lift of the boundary of a hole is embedded by the developing map 
(which we will define in Section~\ref{section:space}) 
and neighborhood of this boundary extends outward, 
then we can fill the hole by gluing the bounded ball in $\bbH^{3}$. 
In general, however, the boundary of a hole may not be embedded by the developing map. 

\fig[width=12cm]{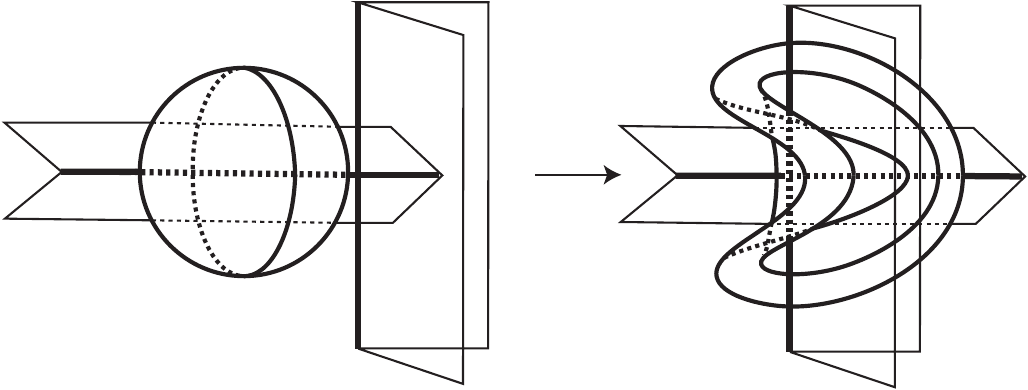}{A non-fillable hole between two cone loci}
 
Cone loci may prevent a hole from expanding to a fillable one. 
This enables us to deform holed cone structures 
when cone structures degenerate by meeting portions of the singular locus. 
The left of Figure~\ref{fig:hcs-intersect.pdf} indicates 
a portion of a fundamental domain for the universal covering map to $X \setminus \Sigma$. 
The vertical and horizontal lines are geodesics. 
Their neighborhoods are portions of a neighborhood of the singular locus. 
The hole is between the two lines. 
The right of Figure~\ref{fig:hcs-intersect.pdf} indicates the image by the developing map. 
The dotted lines indicate the part hidden by the hole. 
The images of the two lines intersect at a point. 
The boundary of the hole is immersed and not embedded by the developing map. 
The hole is not fillable, since otherwise the images of the lines would be disjoint. 
An explicit example will be given in Section~\ref{section:ex}.

\section{Deformation space of holed cone structures}
\label{section:space}

\subsection{Definitions}

Let $\Sigma$ be a link in an oriented 3-manifold $X$, 
and let $\Sigma_{1}, \dots, \Sigma_{n}$ denote the components of $\Sigma$. 
Let $g$ be a (hyperbolic) holed cone metric on $(X,\Sigma)$ with holes $B$. 
Let 
\[
\Gamma = \pi_{1}(X \setminus \Sigma) = \pi_{1}(X \setminus (\Sigma \cup \inter (B))). 
\]
Suppose that $\Gamma$ is finitely generated. 
Note that the holes do not affect the fundamental group. 
Let $\widetilde{M}$ denote the universal cover of 
the incomplete hyperbolic 3-manifold $M = (X \setminus (\Sigma \cup \inter (B)); g)$. 
The \emph{developing map} 
$\dev_{g} \colon \widetilde{M} \to \bbH^{3}$ 
is defined in an ordinary way 
as in \cite{BLP05, CEG87, CHK00}. 

Let $G = \Isom^{+} (\bbH^{3})$ denote 
the group consisting of the orientation-preserving isometries of $\bbH^{3}$. 
The \emph{holonomy representation} 
$\rho_{g} \colon \Gamma \to G$ 
is also defined so that $\dev_{g}$ is equivariant, 
i.e. $\dev_{g}(\gamma \cdot x) = \rho_{g}(\gamma) \cdot \dev_{g}(x)$ 
for any $\gamma \in \Gamma$ and $x \in \widetilde{M}$. 
The representation $\rho_{g}$ is unique up to conjugation in $G$. 
Since the restriction and extension in Definition~\ref{dfn:equivalence} 
do not affect the holonomy representation, 
the holonomy representation $\rho_{g}$ is well-defined for a holed cone structure $[g]$. 
In other words, 
we have $\rho_{g} = \rho_{g^{\prime}}$ for equivalent holed cone metrics $g$ and $g^{\prime}$. 
The developing map is an immersion. 
Conversely, an equivariant immersion induces a metric. 
The restriction of the developing map to the boundary of a hole is an immersion, 
but it is not an embedding in general. 

Consider the representation space $\Hom (\Gamma, G)$ 
endowed with the compact-open topology. 
The group $G$ acts on $\Hom (\Gamma, G)$ by the conjugation. 
Let $[\rho]$ denote the conjugacy class of $\rho \in \Hom (\Gamma, G)$. 
The topological quotient $\Hom (\Gamma, G) / G$ consisting of such $[\rho]$ 
is not Hausdorff in general. 
For instance, 
if $\rho_{0} \colon \Gamma \to \bbZ \to G$ whose non-trivial images are parabolic, 
conjugate elements of $\rho_{0}$ accumulate the trivial representation. 
To obtain a manageable deformation space, 
we need to restrict the space of holonomy representations. 

The group 
$G = \Isom^{+} (\bbH^{3}) \cong \mathrm{PSL}(2, \mathbb{C}) \cong \mathrm{SO}(3, \mathbb{C})$ 
is a complex algebraic group. 
Hence the space $\Hom (\Gamma, G)$ admits a structure of a complex affine algebraic set. 
The \emph{character variety} $\calX (\Gamma) = \Hom (\Gamma,G) / \! / G$ 
is defined as the GIT quotient in the category of algebraic varieties. 
Consider the Euclidean topology of the affine variety $\calX (\Gamma)$, 
which is induced as a subset in $\mathbb{C}^{N}$ for some $N$. 
It is known that the space $\calX (\Gamma)$ is 
the largest Hausdorff quotient of $\Hom (\Gamma, G)$ 
under the coarser relation than that of $\Hom (\Gamma, G) / G$ 
(see \cite{Schwarz89}). 

A representation in $\Hom (\Gamma, G)$ is 
\emph{irreducible} if it is not conjugate to a representation whose image lies in 
$\left\{
\begin{bmatrix}
a & b \\
0 & a^{-1}
\end{bmatrix}
\in \mathrm{PSL}(2, \mathbb{C}) \right\}$. 
Let $\Hom^{\mathrm{irr}} (\Gamma, G)$ denote the set of irreducible representations. 
Let 
\[
t \colon \Hom (\Gamma, G) \to \calX (\Gamma)
\]
denote the projection. 
Any representation $\rho \in \Hom^{\mathrm{irr}} (\Gamma, G)$ satisfies that 
the set $t^{-1} (t(\rho))$ consists of the conjugates of $\rho$ 
(see \cite{BZ98, CS83, HP04, Sikora12} for details). 
Hence we may regard $\calX^{\mathrm{irr}} (\Gamma) = \Hom^{\mathrm{irr}} (\Gamma, G) / G$ 
as a subset of $\calX (\Gamma)$. 
In particular, $\calX^{\mathrm{irr}} (\Gamma)$ is Hausdorff. 
We write $[\rho] = t(\rho) \in \calX^{\mathrm{irr}} (\Gamma)$. 

We define the deformation space of holed cone structures. 
Let $\widetilde{\calHC} (X,\Sigma)$ denote the space of holed cone metrics on $(X,\Sigma)$. 
More precisely, $\widetilde{\calHC} (X,\Sigma)$ is the disjoint union of the spaces of cone metrics on $(X \setminus \inter (B(m)), \Sigma)$ for $m \geq 0$, 
where $B(m)$ is $m$ disjoint balls in $X$, 
and each space is endowed with the $C^{\infty}$-topology induced by the metric tensors. 
Let $\calHC (X,\Sigma)$ denote the space of holed cone structures on $(X,\Sigma)$ 
endowed with the quotient topology from $\widetilde{\calHC} (X,\Sigma)$. 
The continuous map 
\[
\hol \colon \calHC (X,\Sigma) \to \calX (\Gamma) \quad \text{is defined by} \ \hol ([g]) = [\rho_{g}]. 
\]
Let
\[
\widetilde{\calHC}^{\mathrm{irr}}(X,\Sigma) \subset \widetilde{\calHC} (X,\Sigma) 
\quad \text{and} \quad 
\calHC^{\mathrm{irr}} (X,\Sigma) \subset \calHC (X,\Sigma) 
\] 
denote the preimages of $\calX^{\mathrm{irr}} (\Gamma)$ by $\hol([\cdot])$ and $\hol$. 
Note that a quotient space of a Hausdorff space is not Hausdorff in general. 
We show that $\calHC^{\mathrm{irr}}(X,\Sigma)$ is Hausdorff in Theorem~\ref{thm:hausdorff}. 

Let us consider a condition for the holonomy of a neighborhood of the cone loci. 
For each $1 \leq i \leq n$, 
let $N(\Sigma_{i})$ be a regular neighborhood of the cone locus $\Sigma_{i}$ in $X$. 
The \emph{meridian} for $\Sigma$ is (the homotopy class of) 
an essential simple closed curve on $\partial N(\Sigma_{i})$ 
that is contractible in $N(\Sigma_{i})$. 
Fix a \emph{longitude} for $\Sigma$, 
which is a simple closed curve on $\partial N(\Sigma_{i})$ 
intersecting the meridian exactly once. 
Fix orientations of the meridian and longitude 
so that they form a positive basis for the orientation of $\partial N(\Sigma_{i})$ 
induced from that of  $X$. 
Fix commuting elements $\mu_{i}, \lambda_{i} \in \Gamma$ 
which respectively correspond to the meridian and longitude for $\Sigma_{i}$. 
Let $g \in \widetilde{\calHC} (X,\Sigma)$. Suppose that the cone angle at $\Sigma_{i}$ is not zero. 
Then the isometry $\rho_{g}(\lambda_{i}) \in G$ is 
a loxodromic transformation along a geodesic axis $\ell_{i}$. 
(This is not the identity.) 
The isometry $\rho_{g}(\mu_{i}) \in G$ is a rotation about $\ell_{i}$ 
of the cone angle at $\Sigma_{i}$. 
(This is the identity if the cone angle is $2n\pi$, which we allow.) 
The cone locus $\Sigma_{i}$ is added by the metric completion of $g$. 
In the metric completion of the universal cover of $X \setminus (\Sigma \cup \inter(B))$, 
a component of the preimage of $\Sigma_{i}$ is mapped to the geodesic $\ell_{i}$ 
by the continuous extension of the developing map $\dev_{g}$. 
If $\Sigma_{i}$ is a cusp, then $\rho (\lambda_{i})$ and $\rho (\mu_{i})$ are parabolic. 
Moreover, they are conjugate to elements 
acting on $\bbC \subset \partial \bbH^{3}$ 
by linearly independent translations. 
For this condition, we say that $\rho (\lambda_{i})$ and $\rho (\mu_{i})$ are 
\emph{parabolic with rank two}. 

Let $\calX^{\mathrm{cone}}(X,\Sigma)$ denote the set of $[\rho] \in \calX^{\mathrm{irr}}(\Gamma)$ 
satisfying the following conditions for each $1 \leq i \leq n$: 
either $\rho (\lambda_{i})$ and $\rho (\mu_{i})$ are parabolic with rank two, or 
\begin{itemize}
   \item $\rho (\lambda_{i})$ is a loxodromic transformation along an axis, and 
   \item $\rho (\mu_{i})$ is a (possibly trivial) rotation about this axis. 
\end{itemize}
Then the image of $\calHC^{\mathrm{irr}} (X,\Sigma)$ by $\hol$ 
is contained in $\calX^{\mathrm{cone}}(X,\Sigma)$. 
Note that a rotation angle has an arbitrarity modulo $2\pi$. 
We define the space of representations equipped with compatible cone angles. 
Let $\widehat{\calX}^{\mathrm{cone}}(X,\Sigma)$ denote the space consisting of the elements 
$([\rho], \theta_{1}, \dots, \theta_{n}) \in \calX^{\mathrm{cone}}(X,\Sigma) \times [0, \infty)^{n}$ 
such that 
\begin{itemize}
   \item $\theta_{i} = 0$ if $\rho (\lambda_{i})$ and $\rho (\mu_{i})$ are parabolic, 
   \item otherwise $\rho (\mu_{i})$ is a rotation of angle $\theta_{i}$ for the fixed orientation. 
\end{itemize}
Let 
\[
\hat{\pi} \colon \widehat{\calX}^{\mathrm{cone}}(X,\Sigma) \to \calX^{\mathrm{cone}}(X,\Sigma)
\] 
denote the natural projection. 
A fiber of $\hat{\pi}$ consists of elements of the form 
$([\rho], \theta_{1} + 2\pi k_{1}, \dots, \theta_{n} + 2\pi k_{n})$ for some $\theta_{i} \in [0, 2\pi]$, 
where $k_{i} = 0$ if $\theta_{i} = 0$, and otherwise $k_{i}$ is a non-negative integer. 
Define the continuous map 
\[
\widehat{\hol} \colon \calHC^{\mathrm{irr}}(X,\Sigma) \to \widehat{\calX}^{\mathrm{cone}}(X,\Sigma) 
\quad \text{by} \ \widehat{\hol} ([g]) = (\hol ([g]), \Theta (g)), 
\]
where $\Theta (g)$ is $n$-tuple of the cone angles for the metric $g$. 
Then $\hol = \hat{\pi} \circ \widehat{\hol}$. 
Let 
\begin{align*}
\widetilde\calHC^{\mathrm{irr}}_{+} (X,\Sigma) \subset \widetilde{\calHC}^{\mathrm{irr}} (X,\Sigma), 
& \quad 
\calHC^{\mathrm{irr}}_{+} (X,\Sigma) \subset \calHC^{\mathrm{irr}} (X,\Sigma), \\
\calX^{\mathrm{cone}}_{+} (X,\Sigma) \subset \calX^{\mathrm{cone}} (X,\Sigma), 
& \quad
\widehat{\calX}^{\mathrm{cone}}_{+} (X,\Sigma) \subset \widehat{\calX}^{\mathrm{cone}} (X,\Sigma)
\end{align*}
denote the subspaces with positive cone angles. 

The spaces $\calX^{\mathrm{cone}}(X,\Sigma)$ and $\widehat{\calX}^{\mathrm{cone}}(X,\Sigma)$ 
may not be topological manifolds, but their topologies are not very wild. 

\begin{lem}
\label{lem:contract}
The spaces $\calX^{\mathrm{cone}}(X,\Sigma)$ and $\widehat{\calX}^{\mathrm{cone}}(X,\Sigma)$ 
are locally contractible. 
\end{lem}
\begin{proof}
For any $[\rho] \in \calX (\Gamma)$ and $\gamma \in \Gamma$, 
the trace of $\rho (\gamma)$ is determined up to sign. 
The map $[\rho] \mapsto (\tr (\rho (\gamma)))^{2}$ is a regular map on the affine variety $\calX (\Gamma)$. 
The subspace $\calX^{\mathrm{irr}} (\Gamma)$ is open in $\calX (\Gamma)$. 

The condition for the $i$-th meridian to define $\calX^{\mathrm{cone}}(X,\Sigma)$ 
is that the trace of $\rho (\mu_{i})$ belongs to $[-2,2]$. 
Since $\mu_{i}$ and $\lambda_{i}$ commute, 
the condition for the $i$-th longitude is that 
\begin{itemize}
\item $\rho (\lambda_{i})$ and $\rho (\mu_{i})$ are parabolic with rank two if $\rho (\mu_{i})$ is parabolic, 
\item the trace of $\rho (\lambda_{i})$ belongs to $\bbC \setminus [-2,2]$ if $\rho (\mu_{i})$ is elliptic or the identity. 
\end{itemize}
This is an open condition 
under the supposition of the condition for the $i$-th meridian. 

Therefore the space $\calX^{\mathrm{cone}}(X,\Sigma)$ is 
an open subspace of a real semi-algebraic subset of the affine algebraic set $\calX (X,\Sigma)$. 
Each point of a real semi-algebraic set has a neighborhood 
homeomorphic to a cone (in the sense of topology) 
(see \cite[Theorem 5.48]{BPR03}). 
Hence $\calX^{\mathrm{cone}}(\Gamma)$ is locally contractible. 
The space $\widehat{\calX}^{\mathrm{cone}}(X,\Sigma)$ 
is locally homeomorphic to $\calX^{\mathrm{cone}}(X,\Sigma)$. 
\end{proof}

We consider a problem: 
which elements of $\widehat{\calX}^{\mathrm{cone}}(X,\Sigma)$ are realized by holed cone structures? 
For elements of $\widehat{\calX}^{\mathrm{cone}}(X,\Sigma)$, 
we try to construct holed cone structures. 
For this purpose, we introduce the notion of a ``handle decomposition'' for a holed cone metric. 
We define a \emph{handle decomposition} of $(X, \Sigma)$ 
with holes $B$ 
as a filtration $X_{0} \subset X_{1} \subset X_{2} \subset X_{3} = X$ 
of smooth submanifolds satisfying that 
\begin{itemize}
\item $X_{0}$ is the disjoint union of regular neighborhoods of the cone loci and 0-handles, 
\item $X_{i}$ is obtained by attaching $i$-handles to $X_{i-1}$ for $i =1,2,3$, and 
\item $X_{2} \subset X \setminus \inter (B)$. 
\end{itemize}
Write $M_{i} = X_{i} \setminus \Sigma$ for $i =0,1,2$. 
If $g$ is a holed cone metric on $(X, \Sigma)$ with holes $B$, 
then we call a handle decomposition of $(X, \Sigma)$ with holes $B$ 
a handle decomposition for $g$. 
In this case, the restriction of $g$ to $X_{2}$ is also a holed cone metric on $(X, \Sigma)$.

\subsection{Deformations}

We first show that 
if a continuous deformation of holed cone metrics preserves the holonomy representation, 
it also preserves the holed cone structure. 

\begin{lem}
\label{lem:homotopy}
Let $g_{t} \in \widetilde{\calHC}^{\mathrm{irr}} (X,\Sigma)$ be a continuous family 
for $0 \leq t \leq 1$. 
Suppose that $\hol ([g_{t}])$ are constant. 
Then $[g_{t}] \in \calHC^{\mathrm{irr}} (X,\Sigma)$ are constant. 
\end{lem}
\begin{proof}
Let $I_{0}$ be the set consisting of points $s \in [0,1]$ satisfying that 
$[g_{t}] = [g_{0}]$ for any $t \in [0,s]$. 
It is sufficient to show that $I_{0} = [0,1]$. 
Clearly $0 \in I_{0}$. 

Let $s \in [0,1]$. 
We consider a small deformation of holed cone metrics $g_{t}$ 
for $s - \epsilon < t < s + \epsilon$ and some $\epsilon > 0$. 
We take a handle decomposition for each $g_{t}$ 
with a single 0-handle. 
If we take $\epsilon$ sufficiently small, we may assume that 
the handle decompositions for $g_{t}$ are topologically equivalent for $s - \epsilon < t < s + \epsilon$. 
Fix a basepoint in the 0-handle. 
Moreover, we may assume that 
the holonomy representation is constant in the space $\Hom (\Gamma, G)$. 

Since the holonomies of the longitudes of the cone loci are fixed 
to be non-trivial loxodromic or parabolic elements, 
the cone loci are isometrically determined. 
Hence we may assume that the developing maps have identical restrictions 
to the preimage $\widetilde{M_{0}}$ of the space $M_{0} = X_{0} \setminus \Sigma$. 
By replacing them with smaller handles, 
we may assume that the restriction of the developing map for $g_{t}$ to the preimage $\widetilde{M_{2}}$ of $M_{2} = X_{2} \setminus \Sigma$ 
are contained in that for $g_{s}$. 
By taking the projections to $M_{2} = X_{2} \setminus \Sigma$, 
we obtain that $(X_{2}, \Sigma; g_{t})$ are isometrically embedded in $(X_{2}, \Sigma; g_{s})$. 
Thus each $g_{t}$ is equivalent to $g_{s}$. 

The argument for $s \in I_{0}$ implies that $I_{0}$ is open. 
Let $t_{1} = \sup I_{0}$. 
The argument for $s = t_{1}$ implies that $t_{1} = 1 \in I_{0}$. 
Therefore $I_{0} = [0,1]$. 
\end{proof}

We will show that
any fiber of the map 
$\widehat{\hol} \colon \calHC^{\mathrm{irr}}(X,\Sigma) \to \widehat{\calX}^{\mathrm{cone}}(X,\Sigma)$ 
is discrete. 
For this purpose, we introduce an invariant called the \emph{twist} for a holed cone structure 
with respect to a fixed handle decomposition on a holed cone metric of reference. 

We focus on immersed annuli in $\bbR^{3}$. 
We use the notation $\bbZ_{2} = \bbZ / 2\bbZ = \{0,1\}$. 
We say that 
an immersed annulus $f \colon S^{1} \times [0,1] \to \bbR^{3}$ 
bounds an immersed 2-handle 
if there is an immersion of $D^{2} \times [0,1]$ in $\bbR^{3}$ 
whose restriction to $\partial D^{2} \times [0,1]$ coincides with the immersion $f$. 
Here we fix a normal direction on the annulus. 

We use the Smale-Hirsch theorem to describe the regular homotopy classes of handles. 
For manifolds $M$ and $N$, 
let $\Imm (M,N)$ denote the space of immersions from $M$ to $N$ with the $C^{\infty}$-topology, 
and let $\Mon (TM,TN)$ denote the space of vector bundle maps 
from the tangent bundle $TM$ of $M$ to the tangent bundle $TN$ of $N$
that is a monomorphism on each fiber. 
See \cite[Proposition 3.8, Theorem 3.9]{Adachi93} for details. 

\begin{lem}[The Smale-Hirsch fibration lemma]
\label{lem:SH-fibration}
The restriction map 
\[
\Imm (D^{k} \times D^{m-k}, M^{m}) \to 
\Imm (\partial D^{k} \times [0,1] \times D^{m-k}, M^{m})
\] 
induced by the inclusion of a collar of the boundary 
$\partial D^{k} \times [0,1] \hookrightarrow D^{k}$ 
is a Serre fibration for $k < m$. 
\end{lem}

\begin{thm}[The Smale-Hirsch theorem]
\label{thm:SH}
The differential map 
\[
d \colon \Imm (M^{m},N^{n}) \to \Mon (TM^{m},TN^{n}) 
\] 
is a weak homotopy equivalence for $m < n$. 
\end{thm}

\begin{lem}
\label{lem:2handle}
Let $f_{t} \colon S^{1} \times [0,1] \to \bbR^{3}$ be a continuous family of immersions of an annulus 
for $0 \leq t \leq 1$, 
i.e. a regular homotopy between immersed annuli $f_{0}$ and $f_{1}$. 
If $f_{0}$ bounds an immersed 2-handle, 
there is a continuous family of immersed 2-handles bounded by $f_{t}$ 
which starts from this 2-handle. 
\end{lem}
\begin{proof}
Lemma~\ref{lem:SH-fibration} implies that 
the restriction map 
\[
r \colon \Imm (D^{2} \times [0,1], \bbR^{3}) \to \Imm (\partial D^{2} \times [0,1] \times [0,1], \bbR^{3})
\]
induced by the inclusion $\partial D^{2} \times [0,1] \hookrightarrow D^{2}$ 
is a Serre fibration. 
Suppose that the immersed annulus $f_{0}$ bounds an immersed 2-handle. 
In other words, there is an immersion $\tilde{h}_{0} \in \Imm (D^{2} \times [0,1], \bbR^{3})$ 
such that $f_{0}$ is the restriction of $\tilde{h}_{0}$ to $\partial D^{2} \times [0,1] = S^{1} \times [0,1]$. 
Take a path $\{ h_{t} \}_{0 \leq t \leq 1}$ in the base space $\Imm (\partial D^{2} \times [0,1] \times [0,1], \bbR^{3})$ 
such that $r(\tilde{h}_{0}) = h_{0}$ and $h_{t}(x,0,z) = f_{t}(x,z)$ for $x \in \partial D^{2} = S^{1}$ and $z \in [0,1]$. 
The lifting property implies that 
there is a path $\{ \tilde{h}_{t} \}_{0 \leq t \leq 1}$ in $\Imm (D^{2} \times [0,1], \bbR^{3})$ 
such that $r(\tilde{h}_{t}) = h_{t}$. 
Thus we obtain a desired continuous family of immersed 2-handles. 
\end{proof}

We obtain a possibly known result 
for the existence and uniqueness of a 2-handle up to regular homotopy 
(cf. the sphere eversion). 

\begin{lem}
\label{lem:bound}
It holds that 
\[
\pi_{0} (\Imm (S^{1} \times [0,1], \bbR^{3})) \cong \pi_{1}(\mathrm{SO}(3)) \cong \bbZ_{2}. 
\]
An immersed annulus bounds an immersed 2-handle if and only if 
it belongs to the component corresponding to $1 \in \bbZ_{2}$. 
In this case, the regular homotopy class of 2-handles 
with a fixed boundary annulus is unique. 
\end{lem}
\begin{proof}
Theorem~\ref{thm:SH} implies that 
\begin{align*}
\pi_{0} (\Imm (S^{1} \times [0,1], \bbR^{3})) 
& \cong \pi_{0} (\Mon (T(S^{1} \times [0,1]), T\bbR^{3})) \\ 
& \cong \pi_{0} (C^{\infty} (S^{1}, V_{2}(\bbR^{3}))) \\
& \cong \pi_{1} (\mathrm{SO}(3)) \\
& \cong \bbZ_{2}, 
\end{align*}
where $V_{2}(\bbR^{3})$ is the space of 2-frames in $\bbR^{3}$ 
(the non-compact Stiefel manifold), 
and it is homotopy equivalent to $\mathrm{SO}(3)$. 
If an immersed annulus bounds an immersed 2-handle, 
it is regularly homotopic to a small embedded annulus which bounds an embedded 2-handle 
through the immersed 2-handle. 
Then the regular homotopy class corresponds to 
the nontrivial element of $\pi_{1}(\mathrm{SO}(3))$. 
The converse follows from Lemma~\ref{lem:2handle}. 

Theorem~\ref{thm:SH} implies that 
\begin{align*}
\pi_{0} (\Imm (S^{2}, \bbR^{3})) 
& \cong \pi_{0} (\Mon (TS^{2}, T\bbR^{3})) \\ 
& \cong \pi_{0} (C^{\infty} (S^{2}, V_{2}(\bbR^{3}))) \\
& \cong \pi_{2} (\mathrm{SO}(3)) \\
& \cong 0. 
\end{align*}
Hence any two immersed spheres in $\bbR^{3}$ are regularly homotopic. 
In particular, the sphere eversion is possible. 
Suppose that an immersed annulus bounds two immersed 2-handles $H_{1}$ and $H_{2}$. 
By extending $H_{1}$ and $H_{2}$ in a common way, 
we obtain 2-handles $H^{\prime}_{1}$ and $H^{\prime}_{2}$ 
bounded by a common small embedded annulus. 
Let $D_{1}$ and $D_{2}$ be the core disks of $H^{\prime}_{1}$ and $H^{\prime}_{2}$. 
Two immersed spheres obtained by attaching a common small disk to $D_{1}$ and $D_{2}$ 
are regularly homotopic. 
Moreover, we may assume that the homotopy fixes the attached disk. 
Hence the 2-handles $H^{\prime}_{1}$ and $H^{\prime}_{2}$ are regularly homotopic 
fixing the boundary annulus. 
By Lemma~\ref{lem:SH-fibration}, 
this regular homotopy can be deformed to a regular homotopy 
between the 2-handles $H_{1}$ and $H_{2}$ fixing the boundary annulus. 
\end{proof}

Similarly, the following lemma for twists of a 1-handle holds 
(cf. the belt trick). 

\begin{lem}
\label{lem:1handle}
Fix two immersion $i_{0}$ and $i_{1}$ of a disk $D^{2}$ to $\bbR^{3}$. 
Let $\mathcal{I} (i_{0}, i_{1})$ denote the space consisting of the orientation-preserving immersions $f$ of a 1-handle $[0,1] \times D^{2}$ to $\bbR^{3}$ 
such that $f(0, z) = i_{0}(z)$ and $f(1, z) = i_{1}(z)$ for $z \in D^{2}$. 
Then 
\[
\pi_{0} (\mathcal{I} (i_{0}, i_{1})) \cong \pi_{1}(\mathrm{SO}(3)) \cong \bbZ_{2}. 
\]
In other words, there are exactly two regular homotopy classes of 
immersed 1-handles fixing the ends. 
Moreover, a continuous family of immersions from $\{0,1\} \times D^{2}$ to $\bbR^{3}$ 
extends to a continuous family of immersed 1-handles which starts from a given immersed 1-handle. 
\end{lem}
\begin{proof}
The last assertion follows from Lemma~\ref{lem:SH-fibration} 
in the same manner as in the proof of Lemma~\ref{lem:2handle}. 
Since any pair of immersed 1-handles in $\bbR^{3}$ is regulary homotopic to another pair, 
the set $\pi_{0} (\mathcal{I} (i_{0}, i_{1}))$ does not depend on the choice of $i_{0}$ and $i_{1}$. 
Hence we may assume that $i_{0} = i_{1}$. 

Let $I \subset D^{2}$ be a radius. 
The regular homotopy class of an element of $\mathcal{I} (i_{0}, i_{1})$ 
is determined by the regular homotopy class of its restriction to $[0,1] \times I$, 
which induces a map from $S^{1} \times I$. 
Since we can regularly homotope $i_{0} = i_{1}$ to any immersed disk while preserving the regular homotopy class of the corresponding element of $\Imm (S^{1} \times I, \bbR^{3})$, 
we have 
\[
\pi_{0} (\mathcal{I} (i_{0}, i_{1})) \cong \pi_{0} (\Imm (S^{1} \times I, \bbR^{3})) 
\cong \pi_{1} (\mathrm{SO}(3)) \cong \bbZ_{2}, 
\] 
where the second identification was shown in the proof of Lemma~\ref{lem:bound}. 
\end{proof}

We introduce the twist of a holed cone structure of $(X,\Sigma)$ 
with respect to a fixed holed cone metric with a suitable handle decomposition. 
Fix a holed cone metric $g_{0} \in \widetilde{\calHC}^{\mathrm{irr}}_{+} (X,\Sigma)$ 
and a handle decomposition $X_{0} \subset X_{1} \subset X_{2} \subset X_{3} = X$ for $g_{0}$ satisfying that 
\begin{itemize}
\item there is a single 0-handle, 
\item the neighborhood $N(\Sigma_{i}) \subset X_{0}$ of a cone locus $\Sigma_{i}$ is joined to the 0-handle by a single 1-handle, and 
\item the ends of the other 1-handles are attached to the 0-handle. 
\end{itemize}
Recall that $n$ is the number of cone loci. 
Let $m$ denote the number of 1-handles both of whose ends are attached to the 0-handle. 
The space $M_{1} = X_{1} \setminus \Sigma$ induces $2n+m$ generators of the fundamental group $\Gamma$, 
and the 2-handles induce relations. 
Let $g \in \widetilde{\calHC}^{\mathrm{irr}}_{+} (X,\Sigma)$ such that 
$\widehat{\hol} ([g]) = \widehat{\hol} ([g_{0}]) \in \widehat{\calX}^{\mathrm{cone}}_{+} (X,\Sigma)$. 
We define 
\[
\Tw (g; g_{0}) = (x_{1}, \dots, x_{n}, y_{1}, \dots, y_{n}, z_{1}, \dots, z_{m}) 
\in \bbZ^{n} \times (\bbZ_{2}^{n} / \bbZ_{2}) \times \bbZ_{2}^{m}, 
\]
called the \emph{twist} of $g$ with respect to $g_{0}$, 
as follows. 

Take a handle decomposition for $g$ 
which is topologically equivalent to the fixed one for $g_{0}$. 
Let $p \colon \widetilde{M_{2}} \to M_{2} = X_{2} \setminus \Sigma$ denote the universal covering map. 
Let 
\[
\widetilde{M_{1}} = p^{-1}(X_{1} \setminus \Sigma), \ 
\widetilde{M_{0}} = p^{-1}(X_{0} \setminus \Sigma), 
\text{ and } \widetilde{N_{0}(\Sigma_{i})} = p^{-1}(N(\Sigma_{i}) \setminus \Sigma_{i}). 
\]
Consider the restriction of the developing maps $\dev_{g}$ and $\dev_{g_{0}}$ for $g$ and $g_{0}$ to the space $\widetilde{M_{1}}$. 
We continuously deform the metric $g$ preserving the holonomy representation, 
where the holed cone structure $[g]$ is not changed by Lemma~\ref{lem:homotopy}. 
Then we may assume that 
$\dev_{g}$ and $\dev_{g_{0}}$ coincide on the preimage of the 0-handle. 
Moreover, we may assume that 
the restrictions of $\dev_{g}$ and $\dev_{g_{0}}$ to $\widetilde{N_{0}(\Sigma_{i})}$ have a common image. 
However, these restrictions may not coincide. 
The developing maps induce a map $f_{i} \colon (N(\Sigma_{i}); g) \to (N(\Sigma_{i}); g_{0})$ for each $i$. 
The isotopy class of $f_{i}$ is determined by a framing around $\Sigma_{i}$. 
Fix a framing around $\Sigma_{i}$. 
Let $x_{i} \in \bbZ$ denote the twist number of $f_{i}$ measured 
by counting the rotation of the image of the framing by $f_{i}$. 
Note that the cone angles are positive. 
The number $x_{i}$ does not depend on the choice of a framing around $\Sigma_{i}$. 
The numbers $y_{i}, z_{j} \in \bbZ_{2}$ are defined as the regular homotopy classes of the restrictions of the developing maps to the 1-handles 
by Lemma~\ref{lem:1handle}, 
where they are equal to zero if $g = g_{0}$. 
However, a rotation of the 0-handle corresponding to the generator of $\pi_{1}(\mathrm{SO}(3)) \cong \bbZ_{2}$ 
changes the twist in the way that $(y_{1}, \dots, y_{n}) \mapsto (y_{1}+1, \dots, y_{n}+1)$. 
Hence we need to take the quotient by this $\bbZ_{2}$-action on $\bbZ_{2}^{n}$. 
Then the twist $\Tw (g; g_{0})$ is well-defined in $\bbZ^{n} \times (\bbZ_{2}^{n} / \bbZ_{2}) \times \bbZ_{2}^{m}$. 
If two holed cone metrics $g$ and $g'$ satisfy 
$[g] = [g'] \in \calHC^{\mathrm{irr}}_{+} (X,\Sigma)$ and $\widehat{\hol} ([g]) = \widehat{\hol} ([g^{\prime}]) = \widehat{\hol} ([g_{0}])$, 
then $\Tw (g; g_{0}) = \Tw (g'; g_{0})$. 
Clearly $\Tw (g_{0}; g_{0}) = 0$. 

\begin{rem}
\label{rem:cuspframe}
If the cone angle at $\Sigma_{i}$ is equal to zero, 
there is no arbitrarity of the twist number $x_{i}$ of the framing around $\Sigma_{i}$. 
The twist of such a holed cone structure is defined by setting $x_{i} = 0$. 
\end{rem}

Conversely, the twist determines the holed cone structure. 

\begin{lem}
\label{lem:twist}
Let $g, g^{\prime}, g_{0} \in \widetilde{\calHC}^{\mathrm{irr}} (X,\Sigma)$ 
such that $\widehat{\hol} ([g]) = \widehat{\hol} ([g^{\prime}]) = \widehat{\hol} ([g_{0}])$. 
Suppose that $\Tw (g; g_{0}) = \Tw (g^{\prime}; g_{0})$ with respect to a fixed handle decomposition for $g_{0}$. 
Then $[g] = [g^{\prime}]$. 
\end{lem}
\begin{proof}
Fix handle decompositions for $g$ and $g^{\prime}$ 
which are topologically equivalent to that for $g_{0}$. 
We show that there is an equivariant regular homotopy 
between $\dev_{g}|_{\widetilde{M_{2}}}$ and $\dev_{g^{\prime}}|_{\widetilde{M_{2}}}$. 
This corresponds to a continuous deformation of the metrics $g$ to $g^{\prime}$. 
Then Lemma~\ref{lem:homotopy} implies that  $[g] = [g^{\prime}]$. 

Since $\Tw (g; g_{0}) = \Tw (g^{\prime}; g_{0})$, 
we may assume that $\dev_{g}|_{\widetilde{M_{0}}} = \dev_{g^{\prime}}|_{\widetilde{M_{0}}}$. 
Moreover, 
the restrictions of the developing maps to the 1-handles are regularly homotopic. 
This regular homotopy extends to the 2-handles by Lemma~\ref{lem:2handle}. 
Hence we may assume that $\dev_{g}|_{\widetilde{M_{1}}} = \dev_{g^{\prime}}|_{\widetilde{M_{1}}}$. 
Then the attaching annuli of 2-handles for $g$ and $g^{\prime}$ are the same. 
Lemma~\ref{lem:bound} implies that 
$\dev_{g}|_{\widetilde{M_{2}}}$ and $\dev_{g^{\prime}}|_{\widetilde{M_{2}}}$ are regularly homotopic. 
\end{proof}

We consider which elements of $\bbZ^{n} \times (\bbZ_{2}^{n} / \bbZ_{2}) \times \bbZ_{2}^{m}$ are realized as the twists of holed cone structures. 
Fix $g_{0} \in \widetilde{\calHC}^{\mathrm{irr}} (X,\Sigma)$ 
and a handle decomposition for $g_{0}$. 
Let $F(g_{0})$ denote the set consisting of $\Tw (g; g_{0})$ 
for $g$ with $\widehat{\hol} ([g]) = \widehat{\hol} ([g_{0}])$. 
Lemma~\ref{lem:twist} implies that 
the set $F(g_{0})$ is identified with the fiber of the map 
$\widehat{\hol} \colon \calHC^{\mathrm{irr}}(X,\Sigma) \to \widehat{\calX}^{\mathrm{cone}}(X,\Sigma)$ 
over $\widehat{\hol} (g_{0})$. 

\begin{lem}
\label{lem:subgroup}
The set $F(g_{0})$ is a subgroup of 
$\bbZ^{n} \times (\bbZ_{2}^{n} / \bbZ_{2}) \times \bbZ_{2}^{m} 
\cong \bbZ^{n} \times \bbZ_{2}^{n+m-1}$. 
Moreover, the group structure of $F(g_{0})$ 
does not depend on the choice of a handle decomposition for $g_{0}$. 
The group $F(g_{0})$ is infinite unless all the cone angles for $g_{0}$ are equal to zero. 
\end{lem}
\begin{proof}
For simplicity, we suppose that $g_{0} \in \widetilde{\calHC}^{\mathrm{irr}}_{+} (X,\Sigma)$. 
Any element in $\bbZ^{n} \times (\bbZ_{2}^{n} / \bbZ_{2}) \times \bbZ_{2}^{m}$ 
is realized as a twist by a metric on the space $X_{1}$. 
However, this metric is not extended to the 2-handles in general. 
The obstruction is given by Lemma~\ref{lem:bound}. 

Let $D_{1}, \dots, D_{r}$ denote the 2-handles. 
We define a map 
\[
f \colon \bbZ^{n} \times (\bbZ_{2}^{n} / \bbZ_{2}) \times \bbZ_{2}^{m} \to \bbZ_{2}^{r} 
\]
as follows. 
For $1 \leq i \leq n$, $1 \leq j \leq m$, and $1 \leq k \leq r$, 
let $a_{ki}$, $b_{ki}$, $c_{kj}$ denote 
the intersection numbers of the attaching annulus of $D_{k}$ 
and the meridians of $N(\Sigma_{i})$ or the boundaries of the cocores of 1-handles. 
Then $b_{ki}$ is an even number. 
For 
\[
\bm{x} = (x_{1}, \dots, x_{n}, y_{1}, \dots, y_{n}, z_{1}, \dots, z_{m}) 
\in \bbZ^{n} \times (\bbZ_{2}^{n} / \bbZ_{2}) \times \bbZ_{2}^{m}, 
\]
we define $u_{k} \in \bbZ_{2}$ as 
\[
u_{k} \equiv \sum_{i} a_{ki}x_{i} + \sum_{i} b_{ki}y_{i} + \sum_{j} c_{kj}z_{j} \equiv  \sum_{i} a_{ki}x_{i} + \sum_{j} c_{kj}z_{j}. 
\]
The homomorphism $f$ is defined by $f(\bm{x}) = (u_{1}, \dots, u_{r})$. 
Since the element $u_{k}$ represents the twist of the attaching annulus of $D_{k}$ 
measured from that for $g_{0}$, 
it is the obstruction to extend the metric to $D_{k}$. 
Therefore $F(g_{0})$ is the kernel of $f$. 
Thus it is an infinite subgroup of $\bbZ^{n} \times (\bbZ_{2}^{n} / \bbZ_{2}) \times \bbZ_{2}^{m}$. 

The set $F(g_{0})$ is the fiber of $\widehat{\hol}$, 
which is determined as a set. 
Take another handle decomposition 
$X_{0} \subset X'_{1} \subset X'_{2} \subset X_{3} = X$ for $g_{0}$. 
We may assume that $X_{0}$ is the same as above. 
However, 1-handles in $X'_{1}$ may not be homotopic to those in $X_{1}$. 
Similarly to the case for ordinary handle decompositions 
(see \cite[Theorem 1.1]{Kirby89}), 
two handle decompositions 
are connected with a sequence of the following moves and their inverses: 
\begin{itemize}
\item slide of 1-handles and 2-handles, 
\item addition of a 2-handle along a null-homotopic annulus, and 
\item stabilization adding a 1-handle and a 2-handle. 
\end{itemize}
Since slides of handles merely exchange the component of the twist, 
they do not change the group structure of $F(g_{0})$. 
Any addition of a 2-handle does not affect the map $f$. 
Hence we may assume that 
the new handle decomposition is obtained by stabilization adding a 1-handle and a 2-handle. 
Then the obstruction map 
\[
f' \colon \bbZ^{n} \times (\bbZ_{2}^{n} / \bbZ_{2}) \times \bbZ_{2}^{m+1} \to \bbZ_{2}^{r+1}
\] 
for the new handle decomposition is obtained by 
\[
f' (\bm{x}, z_{m+1}) = (f(\bm{x}), z_{m+1}), 
\]
where $\bm{x} = (x_{1}, \dots, x_{n}, y_{1}, \dots, y_{n}, z_{1}, \dots, z_{m})$. 
Therefore the group structures of $F(g_{0})$ for the two handle decompositions coincide. 
\end{proof}

Since the set of twists is discrete, 
the twist invariant is constant under continuous deformation as shown in the following lemma. 

\begin{lem}
\label{lem:twdeform}
Fix a handle decomposition $X_{0} \subset X_{1} \subset X_{2} \subset X_{3} = X$ of $(X, \Sigma)$. 
Let $g_{t}, g'_{t} \in \widetilde{\calHC}^{\mathrm{irr}} (X,\Sigma)$ for $0 \leq t \leq 1$ 
be continuous families of metrics on $X_{2}$ 
such that $\widehat{\hol} ([g_{t}]) = \widehat{\hol} ([g'_{t}])$. 
Then $\Tw (g'_{t}; g_{t})$ are constant for $0 \leq t \leq 1$. 
\end{lem}

We can deform the structures along representations. 
Note that the existence of small deformation is a general result 
(see \cite[Theorem 6.7]{BLP05}, \cite[Lemma I.1.7.2]{CEG87}, or \cite[Theorem 5.3]{CHK00} for details). 

\begin{lem}
\label{lem:lift}
Let $[g_{0}] \in \calHC^{\mathrm{irr}}_{+} (X,\Sigma)$. 
Let $\hat{\rho}_{t} \in \widehat{\calX}^{\mathrm{cone}}_{+} (X,\Sigma)$ for $0 \leq t \leq 1$ 
be a continuous family of representations with positive cone angles 
such that $\widehat{\hol} ([g_{0}]) = \hat{\rho}_{0}$. 
Then there exists a unique continuous family of holed cone structures 
$[g_{t}] \in \calHC^{\mathrm{irr}}_{+} (X,\Sigma)$ starting from $[g_{0}]$ 
such that $\widehat{\hol} ([g_{t}]) = \hat{\rho}_{t}$. 
\end{lem}
\begin{proof}
Take a handle decomposition for $g_{0}$. 
Lemma~\ref{lem:1handle} implies that 
there is a continuous deformation of metrics $g_{t}$ on $X_{1}$ 
such that $\widehat{\hol} ([g_{t}]) = \hat{\rho}_{t}$. 
Note that the cone angles are compatible with the representations 
by definition of $\widehat{\calX}^{\mathrm{cone}}$. 
The developing maps induce a continuous family of the attaching annuli of the 2-handles. 
Then Lemma~\ref{lem:2handle} implies that the metrics extend to the 2-handles. 
Thus we obtain a continuous family of holed cone metrics $g_{t} \in \widetilde{\calHC}^{\mathrm{irr}} (X,\Sigma)$ 
such that $\widehat{\hol} ([g_{t}]) = \hat{\rho}_{t}$. 
The uniqueness of the structures $[g_{t}] \in \calHC^{\mathrm{irr}} (X,\Sigma)$ 
follows from Lemmas~\ref{lem:twist} and \ref{lem:twdeform}. 
\end{proof}

\begin{rem}
\label{rem:cusp}
The assertion of Lemma~\ref{lem:lift} holds 
even if $[g_{0}] \in \calHC^{\mathrm{irr}} (X,\Sigma)$, 
the cone angle at $\Sigma_{i}$ for $[g_{0}]$ is equal to zero, 
and $\hat{\rho}_{t} \in \widehat{\calX}^{\mathrm{cone}}_{+}$ for $0 < t \leq 1$. 
In other words, 
there is a deformation from a cusp to a cone locus along prescribed representations. 
However, there may not be a deformation from a cone locus to a cusp. 
Recall Remark~\ref{rem:cuspframe}. 
Let $[g_{t}] \in \calHC^{\mathrm{irr}} (X,\Sigma)$ for $0 \leq t \leq 1$ 
be a continuous family of holed cone structures. 
Suppose that $[g_{t}] \in \calHC^{\mathrm{irr}}_{+} (X,\Sigma)$ for $0 \leq t < 1$, 
and the cone angle at $\Sigma_{i}$ for $[g_{1}]$ is equal to zero. 
Let $[g^{\prime}_{t}] \in \calHC^{\mathrm{irr}}_{+} (X,\Sigma)$ for $0 \leq t < 1$ 
such that $\widehat{\hol} ([g^{\prime}_{0}]) = \widehat{\hol} ([g_{0}])$. 
Suppose that $\Tw (g^{\prime}_{t}; g_{t})$ has a non-zero entry at $x_{i}$. 
Then $[g^{\prime}_{t}]$ do not converge at $t=1$. 
For $t$ near 1, the metric $g^{\prime}_{t}$ is highly twisted around $\Sigma_{i}$ 
with respect to $g_{t}$. 
\end{rem}

Finally, we obtain the main theorems. 

\begin{thm}
\label{thm:cover}
The map 
$\widehat{\hol} \colon \calHC^{\mathrm{irr}}_{+} (X,\Sigma) \to \widehat{\calX}^{\mathrm{cone}}_{+} (X,\Sigma)$ 
is a regular covering map to each path-connected component 
of $\widehat{\calX}^{\mathrm{cone}}_{+} (X,\Sigma)$ that contains an image. 
\end{thm}
\begin{proof}
For any $[g_{0}] \in \calHC^{\mathrm{irr}}_{+} (X,\Sigma)$, 
we can take a contractible neighborhood $U$ of $\widehat{\hol} ([g_{0}])$ 
in $\widehat{\calX}^{\mathrm{cone}}_{+} (X,\Sigma)$ 
by Lemma~\ref{lem:contract}. 
The unique lifting property by Lemma~\ref{lem:lift} 
implies that there is a section of $\widehat{\hol}$ on $U$ containing $[g_{0}]$. 
Moreover, each fiber of $\widehat{\hol}$ is discrete by Lemma~\ref{lem:subgroup}. 
Hence the image $U_{0}$ of this section is a neighborhood of $[g_{0}]$. 
The group $F(g_{0})$ acts on $\widehat{\hol}^{-1} (U)$ by Lemma~\ref{lem:twdeform}. 
For $\varphi \in F(g_{0})$, let $U_{\varphi} = \varphi \cdot U_{0}$ by this action. 
Then $\widehat{\hol}^{-1} (U)$ is the disjoint union $\bigsqcup_{\varphi \in F(g_{0})} U_{\varphi}$. 
Hence the map $\widehat{\hol}$ is a regular covering map to each path-connected component 
of $\widehat{\calX}^{\mathrm{cone}}_{+} (X,\Sigma)$ containing an image. 
\end{proof}

\begin{thm}
\label{thm:hausdorff}
The deformation space $\calHC^{\mathrm{irr}} (X,\Sigma)$ is Hausdorff. 
\end{thm}
\begin{proof}
By the argument in Remark~\ref{rem:cusp}, 
there is a space $\overline{\calHC}^{\mathrm{irr}} (X,\Sigma)$ 
such that the map $\widehat{\hol} \colon \calHC^{\mathrm{irr}}_{+} (X,\Sigma) \to \widehat{\calX}^{\mathrm{cone}}_{+} (X,\Sigma)$ 
extends to a map $\widehat{\hol} \colon \overline{\calHC}^{\mathrm{irr}} (X,\Sigma) \to \widehat{\calX}^{\mathrm{cone}} (X,\Sigma)$, 
which is a regular covering map to each path-connected component 
of $\widehat{\calX}^{\mathrm{cone}} (X,\Sigma)$ that contains an image. 
The space $\calHC^{\mathrm{irr}} (X,\Sigma)$ is 
an open subset of $\overline{\calHC}^{\mathrm{irr}} (X,\Sigma)$, 
which is Hausdorff. 
\end{proof}

\begin{ques}
\label{ques:product}
Is $\widehat{\hol} \colon \calHC^{\mathrm{irr}}_{+} (X,\Sigma) \to \widehat{\calX}^{\mathrm{cone}}_{+} (X,\Sigma)$ 
a product covering? 
\end{ques}

The map $\widehat{\hol}$ is not a product covering in general. 
For instance, 
let $(X, \Sigma; g)$ and $(X^{\prime}, \Sigma^{\prime}; g^{\prime})$ 
be holed cone-manifolds. 
By making new holes in $(X, \Sigma)$ and $(X^{\prime}, \Sigma^{\prime})$ 
and attaching a 1-handle, 
we obtain a holed cone metric $g_{0}$ 
on the connected sum $(X \# X^{\prime}, \Sigma \sqcup \Sigma^{\prime})$. 
By twisting the 1-handle, 
we obtain a continuous family of holed cone metrics $g_{t}$ 
on $(X \# X^{\prime}, \Sigma \sqcup \Sigma^{\prime})$ for $0 \leq t \leq 1$ 
such that $\widehat{\hol} ([g_{0}]) = \widehat{\hol} ([g_{1}])$ 
and $\Tw (g_{1}; g_{0}) \neq 0$. 
Nonetheless, Question~\ref{ques:product} remains considerable 
in the case that $X \setminus \Sigma$ admits a hyperbolic structure. 

Moreover, let $A_{t}$ be a continuous family of elements in $\Isom^{+} (\bbH^{3})$ for $0 \leq t \leq 1$. 
Let $\rho_{g}$ and $\rho_{g^{\prime}}$ denote the holonomy representations for $(X, \Sigma; g)$ and $(X^{\prime}, \Sigma^{\prime}; g^{\prime})$. 
By deforming metrics on the 1-handle, 
we obtain a continuous family of holed cone metrics on $(X \# X^{\prime}, \Sigma \sqcup \Sigma^{\prime})$ 
such that the holonomy representations of 
\[
\pi_{1}(X \# X^{\prime} \setminus \Sigma \sqcup \Sigma^{\prime}) = \pi_{1}(X \setminus \Sigma) * \pi_{1}(X^{\prime} \setminus \Sigma^{\prime})
\]
are induced by $\rho_{g}$ and $A_{t}^{-1} \rho_{g^{\prime}} A_{t}$. 
Hence there is a continuous family of holed cone metrics on $(X \# X^{\prime}, \Sigma \sqcup \Sigma^{\prime})$ with fixed cone angles. 
In other words, the local rigidity for holed cone structures does not hold in this case. 

By Lemma~\ref{lem:subgroup}, the map $\widehat{\hol} \colon \calHC^{\mathrm{irr}}(X,\Sigma) \to \widehat{\calX}^{\mathrm{cone}}(X,\Sigma)$ is not injective. 
Moreover, there are holed cone structures with a common holonomy representation and different cone angles. 

\begin{prop}
\label{prop:angle}
Let $[g_{0}] \in \calHC^{\mathrm{irr}}_{+} (X,\Sigma)$ 
with cone angles $\theta_{i} > 0$ at $\Sigma_{i}$. 
Suppose that $\theta^{\prime}_{i} = \theta_{i} + 4\pi n_{i} > 0$ for some $n_{i} \in \bbZ$. 
Then there is $[g] \in \calHC^{\mathrm{irr}}_{+} (X,\Sigma)$ 
with cone angles $\theta^{\prime}_{i}$ at $\Sigma_{i}$ 
such that $\hol ([g]) = \hol ([g_{0}])$. 
\end{prop}
\begin{proof}
Take a handle decomposition for $g_{0}$. 
There is a metric $g$ on $X_{1}$ with cone angles $\theta^{\prime}_{i}$ at $\Sigma_{i}$ 
which coincides with $g_{0}$ outside the neighborhoods of the cone loci. 
Then the attaching annulus of each 2-handle for $g$ is obtained by adding an even number of twists (by a sum of integer multiples of $2n_{i}$-rotations) to that for $g_{0}$. 
Since their twist obstructions in Lemma~\ref{lem:bound} coincide, 
the metric $g$ extends to $X_{2}$. 
Thus we obtain $[g] \in \calHC^{\mathrm{irr}}_{+} (X,\Sigma)$ 
with cone angles $\theta^{\prime}_{i}$ at $\Sigma_{i}$ 
such that $\hol ([g]) = \hol ([g_{0}])$. 
\end{proof}

\begin{rem}
The map $\widehat{\hol} \colon \calHC^{\mathrm{irr}}(X,\Sigma) \to \widehat{\calX}^{\mathrm{cone}}(X,\Sigma)$ 
is not surjective in general. 
\end{rem}

Since an oriented 3-manifold is parallelizable, 
any (possibly incomplete) metric on it can be lifted to a spin structure. 
For a hyperbolic metric $g$ on an oriented 3-manifold $M$, 
the holonomy representation $\rho_{g} \colon \Gamma = \pi_{1}(M) \to \mathrm{PSL}(2, \bbC)$ 
has a lift to $\tilde{\rho}_{g} \colon \Gamma \to  \mathrm{SL}(2, \bbC)$, 
which is the holonomy representation of a spin structure (see \cite{Culler86}). 
For a fixed metric $g$, there is a one-to-one correspondence 
between the spin structures over $g$ 
and the group $H^{1}(\Gamma, \bbZ_{2}) = \Hom (\Gamma, \bbZ_{2})$. 
Due to Heusener and Porti~\cite[Theorem 1.4]{HP04}, however, 
there exists an oriented 3-manifold $M$ with boundary consisting of a torus 
such that there are components of the character variety $\calX (\Gamma)$ 
whose elements do not lift to $\mathrm{SL}(2, \bbC)$. 
An element in $\widehat{\calX}^{\mathrm{cone}}(X,\Sigma)$ 
with a representation in such a component 
is not contained in the image of $\widehat{\hol}$. 

For instance, 
let $M$ be a bundle over $S^{1}$ whose fiber is a once-punctured torus $S$. 
Suppose that $M$ is obtained by the action of the monodromy on $H_{1}(S, \bbZ)$ 
given by the matrix
$\begin{pmatrix}
1 & 2 \\
2& 5 
\end{pmatrix}$. 
Then there are presentations 
\[
\pi_{1}(S) = \langle \alpha, \beta \rangle \text{ and } 
\pi_{1}(M) = \langle \alpha, \beta, \mu \mid \mu \alpha \mu^{-1} = \alpha \beta^{2}, \mu \beta \mu^{-1} = \beta (\alpha \beta^{2})^{2} \rangle. 
\]
Let 
\[
A = 
\begin{bmatrix}
i & 0 \\
0 & -i 
\end{bmatrix}, \
B = 
\begin{bmatrix}
1 & 1 \\
-2 & -1 
\end{bmatrix}, \
I = 
\begin{bmatrix}
1 & 0 \\
0 & 1 
\end{bmatrix} 
\in \mathrm{PSL}(2, \bbC). 
\]
By setting $\rho (\alpha) = A$, $\rho (\beta) = B$, and $\rho (\mu) = I$, 
we obtain a representation $\rho \colon \Gamma = \pi_{1}(M) \to \mathrm{PSL}(2, \bbC)$, 
which does not lift to $\mathrm{SL}(2, \bbC)$. 
The element $\alpha \beta \alpha^{-1} \beta^{-1}$ is peripheral in $S$. 
Its image 
\[
ABA^{-1}B^{-1} = 
\begin{bmatrix}
-3 & -2 \\
-4 & -3 
\end{bmatrix}
\]
is loxodromic. 
Let $X$ be the manifold obtained by the Dehn filling on $M$ along the slope corresponding to $\mu$. 
Let $\Sigma$ be the core of the attached solid torus. 
Then $([\rho], 2\pi) \in \widehat{\calX}^{\mathrm{cone}}(X,\Sigma)$. 

\begin{ques}
Let $\hat{\rho} = (\rho, \theta_{1}, \dots, \theta_{n}) \in \widehat{\calX}^{\mathrm{cone}}(X,\Sigma)$. 
Suppose that $\rho$ has a lift to a representation to $\mathrm{SL}(2, \bbC)$. 
Then is $\hat{\rho}$ contained in the image of the map 
$\widehat{\hol} \colon \calHC^{\mathrm{irr}}(X,\Sigma) \to \widehat{\calX}^{\mathrm{cone}}(X,\Sigma)$? 
\end{ques}

\section{Cone structures in the deformation space}
\label{section:cone}

In this section, 
we consider (non-holed hyperbolic) cone structures in the space of holed cone structures. 
Recall that a cone structure is an equivalence class of cone metrics, 
where the equivalence relation is induced by isometry isotopic to the identity. 
Suppose that there are cone structures on $(X, \Sigma)$. 
Due to Kojima~\cite[Theorem 1.2.1]{Kojima98}, 
this is equivalent to that $X \setminus \Sigma$ admits a hyperbolic structure. 
In particular, the 3-manifold $X \setminus \Sigma$ is irreducible. 

\begin{thm}
\label{thm:conehol}
Let $g$ and $g^{\prime}$ be cone metrics on $(X, \Sigma)$. 
Suppose that $\widehat{\hol} ([g]) = \widehat{\hol} ([g^{\prime}])$. 
Then $g$ and $g^{\prime}$ are equivalent as cone metrics. 
\end{thm}

To show the theorem, we need the following lemmas. 
Take a handle decomposition $X_{0} \subset X_{1} \subset X_{2} \subset X_{3} = X$ of $(X, \Sigma)$. 

\begin{lem}
\label{lem:extend}
Let $g$ and $g^{\prime}$ be cone metrics on $(X, \Sigma)$. 
If there is an isometry $\iota$ from $(X_{1} \setminus \Sigma; g)$ to $(X_{1} \setminus \Sigma; g')$, 
then $\iota$ extends to an isometry from $(X \setminus \Sigma; g)$ to $(X \setminus \Sigma; g')$. 
\end{lem}
\begin{proof}
Let $M = X \setminus \Sigma$ and $M_{1} = X_{1} \setminus \Sigma$. 
Let $\widetilde{M}$ denote the universal cover of $M$, 
and let $\widetilde{M_{1}} \subset \widetilde{M}$ denote the preimage of $M_{1} \subset M$. 
Since the natural homomorphism $\pi_{1} (M_{1}) \to \pi_{1} (X \setminus \Sigma)$ 
is surjective, 
the space $\widetilde{M_{1}}$ is connected. 
The metrics $g$ and $g^{\prime}$ induce the holonomy representations $\rho_{g}$ and $\rho_{g^{\prime}}$. 
We have $\rho_{g^{\prime}} = \rho_{g} \circ \iota_{*}$. 
The developing maps $\dev_{g}, \dev_{g^{\prime}} \colon \widetilde{M} \to \bbH^{3}$ 
satisfy that $\dev_{g^{\prime}}|_{\widetilde{M_{1}}} = \dev_{g^{\prime}}|_{\widetilde{M_{1}}} \circ \tilde{\iota}$, 
where $\tilde{\iota} \colon (\widetilde{M_{1}}; g) \to (\widetilde{M_{1}}; g^{\prime})$ 
is the isometry induced by $\iota$. 
We extend $\tilde{\iota}$ to an isometry from $(\widetilde{M}; g)$ to $(\widetilde{M}; g^{\prime})$ 
by connecting local isometries along a path from $\widetilde{M_{1}}$. 
Since $\widetilde{M}$ is simply connected, the extension is well-defined, 
that is, it does not depend on the ways of taking paths. 
The subset of $\widetilde{M}$ consisting of the points to which the isometry extends 
is closed and open. 
Hence we can extend $\tilde{\iota}$ to an isometry from the whole of $\widetilde{M}$. 
The uniqueness of the extension implies that 
the isometry $\tilde{\iota} \colon (\widetilde{M}; g) \to (\widetilde{M}; g^{\prime})$ is equivariant 
with respect to $\rho_{g}$ and $\rho_{g^{\prime}}$. 
By taking the projections, 
we obtain an isometry from $(M; g)$ to $(M; g^{\prime})$ which is an extension of $\iota$. 
\end{proof}

\begin{lem}
\label{lem:homeo}
Suppose that there is an isometry between cone metrics 
$(X, \Sigma; g)$ and $(X, \Sigma; g^{\prime})$ 
which induces the identity on the fundamental group of $X \setminus \Sigma$. 
Then $g$ and $g^{\prime}$ are equivalent as cone metrics. 
\end{lem}
\begin{proof}
Let $\Homeo (X \setminus \Sigma)$ denote the space of homeomorphisms of $X \setminus \Sigma$ 
with the compact-open topology. 
Let $\Out (\Gamma)$ denote 
the outer automorphism group of $\Gamma = \pi_{1} (X \setminus \Sigma)$. 
The action of homeomorphisms on the fundamental group $\Gamma$ induces 
the natural homomorphism $j \colon \Homeo (X \setminus \Sigma) \to \Out (\Gamma)$. 
Isotopic self-homeomorphisms of $X \setminus \Sigma$ are mapped by $j$ to a common element of $\Out (\Gamma)$. 
Let $i \colon \Isom (X \setminus \Sigma) \to \Homeo (X \setminus \Sigma)$ 
denote the inclusion map for a fixed hyperbolic structure on $X \setminus \Sigma$. 
The Mostow rigidity implies that 
every self-homeomorphism of $X \setminus \Sigma$ is isotopic to an isometry. 
In other words, every path-connected component of $\Homeo (X \setminus \Sigma)$ 
contains an element of $\Isom (X \setminus \Sigma)$. 
Moreover, $j \circ i \colon  \Isom (X \setminus \Sigma) \to \Out (\Gamma)$ is an isomorphism 
by \cite[Theorem C.5.6]{BP92}. 
Hence the kernel of $j$ is a single path-connected component of $\Homeo (X \setminus \Sigma)$. 
Therefore the isometry between $(X, \Sigma; g)$ and $(X, \Sigma; g^{\prime})$ is isotopic to the identity. 
\end{proof}

\begin{proof}[Proof of Theorem~\ref{thm:conehol}]
Since $\widehat{\hol} ([g]) = \widehat{\hol} ([g^{\prime}])$, 
we may assume that 
there is an isometry $(X_{0} \setminus \Sigma; g)$ to $(X_{0} \setminus \Sigma; g^{\prime})$. 
By deforming $g^{\prime}$ by isotopy, 
we extend this isometry to $X_{1} \setminus \Sigma$. 
Then we obtain an isometry from $(X_{1} \setminus \Sigma; g)$ to $(X_{1} \setminus \Sigma; g^{\prime})$. 
Hence there is an isometry from $(X \setminus \Sigma; g)$ to $(X \setminus \Sigma; g^{\prime})$ 
by Lemma~\ref{lem:extend}. 
Since $\widehat{\hol} ([g]) = \widehat{\hol} ([g^{\prime}])$, 
this isometry induces the identity on the fundamental group. 
Hence $g$ and $g^{\prime}$ are equivalent as cone metrics by Lemma~\ref{lem:homeo}. 
\end{proof}

\begin{cor}
\label{cor:coneeq}
Two cone metrics $g$ and $g^{\prime}$ on $(X,\Sigma)$ are equivalent as holed cone metrics 
if and only if they are equivalent as cone metrics. 
\end{cor}
\begin{proof}
The ``if'' part is trivial. 
Suppose that $g$ and $g^{\prime}$ are equivalent as holed cone metrics. 
Then $\widehat{\hol} ([g]) = \widehat{\hol} ([g^{\prime}])$. 
Hence $g$ and $g^{\prime}$ are equivalent as cone metrics by Theorem~\ref{thm:conehol}. 
\end{proof}

Let $\calC (X, \Sigma)$ denote the space of hyperbolic cone structures on $(X, \Sigma)$. 
For $[g] \in \calC (X, \Sigma)$, 
the holonomy representation $\rho_{g}$ is irreducible \cite[Lemma 4.6]{HK98}. 
Corollary~\ref{cor:coneeq} implies that 
we may regard $\calC (X, \Sigma) \subset \calHC^{\mathrm{irr}} (X, \Sigma)$. 
Therefore a holed cone structure is a generalization of a cone structure. 

The space $\calC (X, \Sigma)$ is an open subspace of $\calHC^{\mathrm{irr}} (X, \Sigma)$ 
because there is small deformation of cone structures 
(see \cite[Theorem 6.7]{BLP05}, \cite[Lemma I.1.7.2]{CEG87}, or \cite[Theorem 5.3]{CHK00} for details). 
The generalized hyperbolic Dehn surgery around the cusped hyperbolic structure $[g_{c}] \in \calC (X, \Sigma)$ induces a neighborhood of $\widehat{\calX}^{\mathrm{cone}} (\Gamma)$ of the discrete faithful representation, which is parametrized by small cone angles. 

The global structure of the space $\calC (X, \Sigma)$ is not well known. 
For instance, it is not known whether $\calC (X, \Sigma)$ is path-connected. 
Theorem~\ref{thm:conehol} implies that the map 
$\widehat{\hol} \colon \calC (X, \Sigma) \to \widehat{\calX}^{\mathrm{cone}} (\Gamma)$ 
is injective. 
Here one may ask whether the condition of cone angles is necessary. 
While a holed cone structure is a generalization of a cone structure, 
an element in the components of $\calHC^{\mathrm{irr}} (X, \Sigma)$ containing $\calC (X, \Sigma)$ 
can be regarded as a moderate generalization. 
If we can fill holes, the generalization is more meaningful. 
Even in the case that $\calC (X, \Sigma)$ is not path-connected, 
it is convenient if there is a continuous deformation between two cone structures via holed cone structures. 
From these perspectives, we ask the following questions. 

\begin{ques}
Is the map $\hol \colon \calC (X, \Sigma) \to \calX^{\mathrm{cone}} (\Gamma)$ injective? 
In other words, 
are there cone-manifolds with a common holonomy representation and distinct cone angles? 
\end{ques}

\begin{ques}
Let $[g_{t}] \in \calHC^{\mathrm{irr}} (X, \Sigma)$ be a continuous family for $0 \leq t \leq 1$. 
Suppose that $[g_{0}] \in \calC (X, \Sigma)$. 
Suppose that there is $[g'] \in \calC (X, \Sigma)$ such that 
$\widehat{\hol} ([g']) = \widehat{\hol} ([g_{1}])$. 
Then $[g']= [g_{1}]$? 
Equivalently, is the map $\widehat{\hol}$ injective 
on the union of components of $\calHC^{\mathrm{irr}} (X, \Sigma)$ 
containing $\calC (X, \Sigma)$? 
\end{ques}

\begin{ques}
Can a cone structure in $\calC (X, \Sigma)$ be deformed to 
the cusped hyperbolic structure on $X \setminus \Sigma$ 
via holed cone structures? 
Equivalently, 
is $\calC (X, \Sigma)$ contained in a single component of $\calHC^{\mathrm{irr}} (X, \Sigma)$? 
\end{ques}

\section{Volumes of holed cone structures}
\label{section:vol}

In this section, we introduce the volume of a holed cone structure. 
To obtain a well-defined value, 
we add the volume enclosed by the holes 
to the volume of a holed cone metric. 

Let $f \colon S^{2} \to \bbH^{3}$ be an immersion. 
Fix a normal orientation on $f(S^{2})$. 
We define the signed enclosed volume $V(f) \in \bbR$ as follows. 
Suppose that $f$ is generic. 
Let $R_{0}, R_{1}, \dots, R_{N}$ denote the connected components 
of the complement $\bbH^{3} \setminus f(S^{2})$, 
where $R_{0}$ is the unbounded component. 
Each component $R_{i}$ is assigned an integer $\deg (R_{i})$ so that 
\begin{itemize}
\item $\deg (R_{0}) = 0$, and 
\item $\deg (R_{j}) = \deg (R_{i}) + 1$ 
         if $R_{i}$ and $R_{j}$ are adjacent 
         and the normal orientation on $f(S^{2})$ has direction from $R_{i}$ to $R_{j}$. 
\end{itemize}
Then we define $V(f) = \sum_{i = 1}^{N} \deg (R_{i}) \vol (R_{i})$. 
The function 
\[
V \colon \Imm (S^{2}, \bbH^{3}) \to \bbR
\]
is defined by continuous extension. 
Recall that $\Imm (S^{2}, \bbH^{3})$ is the space of immersions from $S^{2}$ to $\bbH^{3}$ 
with the $C^{\infty}$ topology. 
The following lemma verifies the definition. 

\begin{lem}
\label{lem:deg}
Let $F$ be a smooth map from the 3-ball $B^{3}$ to $\bbH^{3}$ 
whose restriction to $\partial B^{3} = S^{2}$ is a generic immersion $f \in \Imm (S^{2}, \bbH^{3})$. 
The orientation on $B^{3}$ induces the normal inward orientation on $f(S^{2})$. 
Write $\bbH^{3} \setminus f(S^{2}) = R_{0} \sqcup R_{1} \sqcup \dots \sqcup R_{N}$ 
as above. 
Then the map $F$ has the mapping degree $\deg (R_{i})$ 
on each $R_{i}$ for $0 \leq i \leq N$. 
Moreover, we have 
$\int_{B^{3}} F^{*} \omega = V(f)$, 
where $\omega$ is the volume form on $\bbH^{3}$. 
\end{lem}
\begin{proof}
Note that any immersion $f$ has such an extension $F$ 
since the space $\bbH^{3}$ is contractible. 
The mapping degree of $F$ on a regular value $p$ in $\bbH^{3} \setminus f(S^{2})$ 
is defined as the sum of $\pm 1$'s over $F^{-1}(p)$, 
where the sign is positive if and only if $F$ preserves the orientation around the point. 
The mapping degree of $F$ is constant on each $R_{i}$. 
Since the image of $F$ is bounded, we have $\deg (R_{0}) = 0$. 
The choice of the orientation implies that 
the mapping degrees of $F$ satisfy the same relation for $\deg (R_{i})$. 
Hence they coincide. 
The definition of induced form implies that 
\[
\int_{B^{3}} F^{*} \omega = \sum_{i=1}^{N} \int_{R_{i}} \deg (R_{i}) \omega 
= \sum_{i = 1}^{N} \deg (R_{i}) \vol (R_{i}) = V(f). 
\]
\end{proof}

An explicit formula for $V(f)$ may be useful 
for differential geometric approaches. 

\begin{ques}
Is there an explicit formula for $V(f)$? 
Is it possible to express it by the curvatures of $f$? 
\end{ques}

Let $g \in \widetilde{\calHC}^{\mathrm{irr}} (X,\Sigma)$
and $\sigma = [g] \in \calHC^{\mathrm{irr}} (X,\Sigma)$. 
Let $B_{1}, \dots, B_{m}$ denote the holes of $g$. 
For $1 \leq i \leq m$, define an immersion $f_{i} \colon S^{2} \to \bbH^{3}$ 
as the restriction of the developing map $\dev_{g}$ to a lift of $\partial B_{i}$. 
Fix the normal orientation of $f_{i} (S^{2})$ towards the interior of $B_{i}$. 
We define the volume of the holed cone structure $\sigma$ as 
$\vol (\sigma) = \vol (g) + \sum_{i=1}^{m} V(f_{i})$, 
where $\vol (g)$ is the volume of the holed cone metric $g$ as a Riemannian metric. 
Let $g^{\prime} \in \widetilde{\calHC}^{\mathrm{irr}} (X,\Sigma)$ be a restriction of $g$ 
with holes $B^{\prime}_{1}, \dots, B^{\prime}_{m^{\prime}}$. 
Let $f^{\prime}_{j}$ be immersions for $B^{\prime}_{j}$ as above. 
Then $\vol (g) - \vol (g^{\prime}) = \sum_{j=1}^{m^{\prime}} V(f^{\prime}_{j}) - \sum_{i=1}^{m} V(f_{i})$ 
by the definition of the signed enclosed volume. 
Hence $\vol (\sigma)$ is well-defined. 

Let $\rho \colon \Gamma = \pi_{1} (X \setminus \Sigma) \to \mathrm{PSL}(2, \bbC)$ 
be a representation. 
Due to Dunfield~\cite{Dunfield99}, 
the volume of the representation $\rho$ is defined as follows. 
Let $\widetilde{M} \to M = X \setminus \Sigma$ denote the universal covering. 
Let $\widehat{M}$ and $\widehat{\widetilde{M}}$ respectively denote 
the compactifications of $M$ and $\widetilde{M}$ 
such that each end is compactified by adding one point. 
The added points are called ideal points. 
Let $\overline{\bbH^{3}} = \bbH^{3} \cup \partial \bbH^{3}$ 
denote the natural compactification of $\bbH^{3}$ 
by adding the sphere at infinity. 

The singular locus $\Sigma_{i}$ corresponds to an ideal point $v_{i}$ of $\widehat{M}$. 
We fix a product structure of a neighborhood $N_{i}$ of $v_{i}$ 
which is obtained from $T_{i} \times [0, \infty]$ 
by collapsing $T_{i} \times \{ \infty \}$ to $v_{i}$. 
We lift this product structure to the ends of $\widetilde{M}$. 
Then an ideal point $\tilde{v}$ of $\widehat{\widetilde{M}}$ 
has a neighborhood 
$\widetilde{N}_{\tilde{v}} = 
P_{\tilde{v}} \times [0, \infty] /  (P_{\tilde{v}} \times \{ \infty \})$, 
where $P_{\tilde{v}}$ covers a torus $T_{i}$ for some $i$. 
For a set $A$, 
let $C = A \times [c, \infty] / (A \times \{ \infty \})$ 
by collapsing $A \times \{ \infty \}$ to a point $\infty$. 
A map $f \colon C \to \overline{\bbH^{3}}$ is a \emph{cone map} if 
\begin{itemize}
\item $f(C) \cap \partial \bbH^{3} = \{ f(\infty) \}$, and 
\item for any $a \in A$, the map $f|_{a \times [c, \infty]}$ 
         is the geodesic ray from $f(a, c)$ to $f(\infty)$ 
         parametrized by the arc length. 
\end{itemize}
A \emph{pseudo-developing map} for $\rho$ 
is a piecewise smooth map $D_{\rho} \colon \widetilde{M} \to \bbH^{3}$ 
satisfying the following conditions: 
\begin{itemize}
\item $D_{\rho}$ is equivariant for the $\Gamma$-actions, 
         which are the deck transformation on $\widetilde{M}$ 
         and the action on $\bbH^{3}$ via $\rho$. 
\item $D_{\rho}$ extends continuously to a map 
         $\widehat{D}_{\rho} \colon \widehat{\widetilde{M}} \to \overline{\bbH^{3}}$. 
\item There exists $c \geq 0$ such that 
         the restriction of $\widehat{D}_{\rho}$ to each end 
         $P_{\tilde{v}} \times [c, \infty] /  (P_{\tilde{v}} \times \{ \infty \})$
         is a cone map. 
\end{itemize}

Let $\tilde{v}_{i}$ an ideal point of $\widehat{\widetilde{M}}$ 
which is a lift of $v_{i} \in \widehat{M}$. 
Recall that $\mu_{i}, \lambda_{i} \in \Gamma$ 
are commuting elements corresponding to the meridian and the longitude for $\Sigma_{i}$. 
The stabilizer $\Stab (\tilde{v}_{i}) < \Gamma$ of $\tilde{v}_{i}$ 
is generated by conjugates of $\mu_{i}$ and $\lambda_{i}$. 
The fixed point set $\Fix (\rho (\Stab (\tilde{v}_{i}))) \subset \partial \bbH^{3}$ 
consists of one or two points 
depending on whether the cone angle at $\Sigma_{i}$ is zero or positive respectively. 
If it is positive, 
$\Fix (\rho (\Stab (\tilde{v}_{i})))$ consists of the endpoints of a lift of $\Sigma_{i}$. 
Then $\widehat{D}_{\rho} (\tilde{v}_{i})$ is contained in $\Fix (\rho (\Stab (\tilde{v}_{i})))$. 

Let $\omega$ denote the volume form of $\bbH^{3}$. 
Since the 3-form $D_{\rho}^{*} \omega$ on $\widetilde{M}$ is equivariant, 
it projects to a 3-form (also denoted by $D_{\rho}^{*} \omega$) on $M$. 
Define the volume of the representation $\rho$ 
as $\vol (\rho) = \vol (D_{\rho}) = \int_{M} D_{\rho}^{*} \omega$. 
Francaviglia~\cite{Francaviglia04} proved that 
$\vol (\rho)$ does not depend on the choice of a pseudo-developing map. 
We refer Kim~\cite{Kim15} 
for various equivalent definitions of the volume of a representation. 
We show that the volume of a holed cone structure 
coincides with the volume of the holonomy representation. 
The assertion for cone-manifolds was shown by Porti~\cite{Porti98}. 

\begin{thm}
\label{thm:vol}
Let $\sigma \in \calHC^{\mathrm{irr}} (X,\Sigma)$. 
Then $\vol (\sigma) = \vol (\hol (\sigma))$. 
\end{thm}
\begin{proof}
Take $g \in \widetilde{\calHC}^{\mathrm{irr}} (X,\Sigma)$ such that $[g] = \sigma$. 
Let $\rho = \hol (\sigma)$. 
For simplicity, we assume that the hole $B$ of $g$ is a single ball. 
Let $f$ be the restriction of $\dev_{g}$ to a lift of $\partial B$. 
We extend $f$ to $F \colon B \to \bbH^{3}$ as in Lemma~\ref{lem:deg}. 

Let $N(\Sigma_{i})$ be a regular neighborhood of $\Sigma_{i}$ 
which is a standard tube or cusp with respect to $g$. 
Let $\widehat{\widetilde{N}}_{i}$ denote the one-point compactification 
of the universal cover of $N(\Sigma_{i}) \setminus \Sigma_{i}$. 
The restriction $\dev_{g}$ to the preimage of $\partial N(\Sigma_{i})$ 
extends to a cone map from $\widehat{\widetilde{N}}_{i}$, 
obtained by joining each point of the boundary to a point in $\Fix (\rho (\Stab (\tilde{v}_{i})))$. 
This cone map is equivariant and 
projected to a self-homeomorphism on $N(\Sigma_{i}) \setminus \Sigma_{i}$. 

We construct a pseudo-developing map $D_{\rho}$ 
by combining the restriction $\dev_{g}$ 
to the preimage of $X \setminus (\bigcup_{i} N(\Sigma_{i}) \cup B)$, 
the cone maps from $\widehat{\widetilde{N}}_{i}$, and copies of $F$. 
Then we have $\int_{X \setminus B} D_{\rho}^{*} \omega = \vol (g)$ 
and $\int_{B} D_{\rho}^{*} \omega = V(f)$ by Lemma~\ref{lem:deg}. 
Hence $\vol (\rho) = \int_{X \setminus \Sigma} D_{\rho}^{*} \omega = \vol (g) + V(f) = \vol (\sigma)$. 
\end{proof}

The Schl\"{a}fli formula for holed cone structures is expressed as follows. 
This also follows from Hodgson's form for representations \cite{Hodgson86}. 

\begin{prop}
Let $\sigma_{t} \in \calHC^{\mathrm{irr}} (X,\Sigma)$ 
be a continuous family for $-\epsilon < t < \epsilon$. 
Then 
\[
\dfrac{d}{dt} \vol (\sigma_{t}) = 
-\dfrac{1}{2} \sum_{i=1}^{n} \ell_{i}(t) \dfrac{d}{dt} \theta_{i}(t), 
\] 
where $\ell_{i}(t)$ is the length of $\Sigma_{i}$ for $\sigma_{t}$ 
and $\theta_{i}(t)$ is the cone angle at $\Sigma_{i}$ for $\sigma_{t}$. 
\end{prop}
\begin{proof}
There is a continuous family 
$g_{t} \in \widetilde{\calHC}^{\mathrm{irr}} (X,\Sigma)$ for $-\epsilon < t < \epsilon$ 
with a single hole $B$ 
such that $\sigma_{t} = [g_{t}]$. 
By taking sufficiently small $\epsilon$, 
we may assume that 
the restrictions of $\dev_{g_{t}}$ to each lift of $\partial B$ 
are deformed by isometries of $\bbH^{3}$. 
Then $\dfrac{d}{dt} \vol (\sigma_{t}) = \dfrac{d}{dt} \vol (g_{t})$. 

Take a triangulation of $X \setminus \inter (B)$ 
so that $\Sigma$ is contained in the 1-skeleton. 
We may assume that 
the simplices adjacent to $\partial B$ have constant metrics for $t$, 
and the other simplices are geodesic. 
The formula for a hyperbolic tetrahedron 
is expressed by the lengths of edges and dihedral angles 
in the same form as the assertion (see \cite{Porti98}). 
The assertion follows by taking the sum of terms for the tetrahedra. 
Since the sum of dihedral angles about each internal edge is $2\pi$, 
the corresponding terms in the formula cancel. 
\end{proof}

\section{Example}
\label{section:ex}

Let $L = L_{1} \sqcup \dots \sqcup L_{4}$ be a link in $X = T^{2} \times I$ 
as indicated on the left of Figure~\ref{fig:hcs-decomp.pdf}, 
where $I$ is an open interval. 
This is a quotient of the ``plain weave.'' 
The author~\cite{Yoshida22} described hyperbolic cone structures on $(X,L)$, 
and gave an example of their degeneration with decreasing cone angles. 
Cone loci meet in this degeneration. 
Nonetheless, holed cone structures enable us to avoid such degeneration. 

\fig[width=12cm]{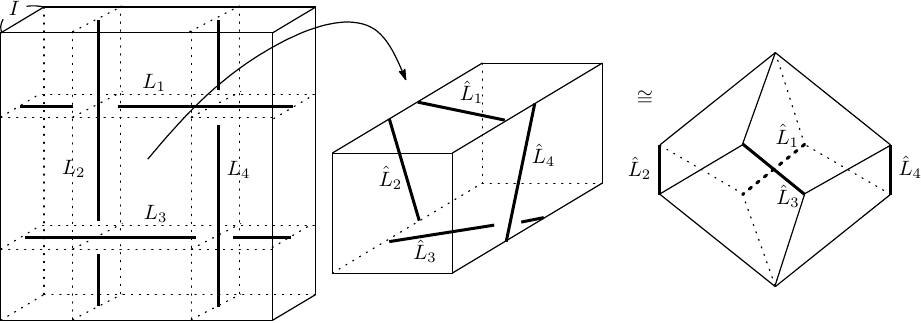}{Decomposition of $(X,L)$ into trapezohedra}

The space $(X,L)$ is decomposed into four polyhedra, called \emph{(tetragonal) trapezohedra}, 
as indicated in Figure~\ref{fig:hcs-decomp.pdf}. 
Conversely, we can construct a hyperbolic cone structure on $(X,L)$ 
by gluing four copies of a hyperbolic trapezohedron. 
For each $i= 1, \dots , 4$, the cone locus $L_{i}$ derives from the edge $\hat{L}_{i}$. 
Suppose that the dihedral angle at $\hat{L}_{i}$ is $\alpha_{i}$. 
The dihedral angle at any other edge needs to be a right angle. 
If $\alpha_{i} = 0$, then $\hat{L}_{i}$ degenerates to an ideal vertex. 
Nonetheless, we continue to call it a trapezohedron. 
Then the cone angle at $L_{i}$ is $2\alpha_{i}$. 
It is not known whether every cone structure on $(X,L)$ 
is obtained by this construction. 
However, this is true if the global rigidity for the cone structures on $(X,L)$ holds. 

Extending this, 
we construct a holed hyperbolic cone metric on $(X,L)$ by gluing ``holed trapezohedra''. 
As a holed trapezohedron, 
we consider the complement of holes in a trapezohedron 
endowed with a hyperbolic metric 
such that 
\begin{itemize}
\item each hole is disjoint from the edges $\hat{L}_{i}$ and the vertices, 
\item each face is totally geodesic, and 
\item the boundary of holes is orthogonal to faces. 
\end{itemize}
A holed trapezohedron is isometrically immersed in $\bbH^{3}$ 
as indicated in Figure~\ref{fig:hcs-holed.pdf}. 
The boundary of a hole is drawn in blue. 

\fig[width=12cm]{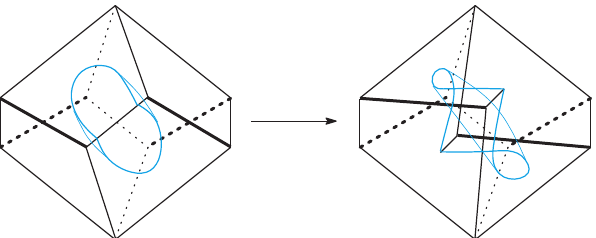}{A holed trapezohedron}

Let us consider a holed trapezohedron 
such that the cone angles at the edges other than $\hat{L}_{i}$ 
are right angles. 
We can glue four copies of such a holed trapezohedron 
in the same manner as for non-holed ones. 
Then we obtain a holed cone metric on $(X,L)$. 

According to the description in \cite{Yoshida22}, 
let 
\begin{align*}
\mathcal{B} &= 
\{(q_{1}, \dots , q_{4},t) \in \mathbb{R}_{>0}^{4} \times \mathbb{R}_{\geq 0} 
\mid \prod_{i=1}^{4} q_{i} = 1, \
t \geq \frac{1}{2}(q_{i} - q_{i}^{-1})\}, \\
\mathcal{B}_{0} &= \{(q_{1}, \dots , q_{4},t) \in \mathcal{B} 
\mid (1-q_{i}q_{i+1}) t < q_{i} + q_{i+1} \} 
\end{align*}
with the indices $i = 1, \dots, 4$ modulo 4. 
Define $f \colon \mathcal{B} \to \mathbb{R}^{4}$ by 
\[
f(q_{1}, \dots , q_{4},t) = 
\left( \frac{q_{1}-t}{\sqrt{1+t^{2}}}, \dots , \frac{q_{4}-t}{\sqrt{1+t^{2}}} \right). 
\]
Then $f$ is a homeomorphism onto $(-1,1]^{4}$. 
There exists a trapezohedron such that 
the dihedral angle at $\hat{L}_{i}$ is $\alpha_{i}$ 
and the dihedral angle at any other edge is the right angle 
if and only if 
\[
(\cos \alpha_{1}, \dots , \cos \alpha_{4}) \in f(\mathcal{B}_{0}).
\] 

We show that every element in $\mathcal{B}$ corresponds to a holed trapezohedron. 
Because the construction is consistent with that in \cite{Yoshida22}, 
we briefly sketch it. 

\begin{thm}
\label{thm:htrap}
For any $(\alpha_{1}, \dots, \alpha_{4}) \in [0, \pi)^{4}$, 
there exists a holed trapezohedron such that 
the cone angle at $\hat{L}_{i}$ is $\alpha_{i}$ 
and the cone angle at any other edge is a right angle. 
\end{thm}
\begin{proof}
We use the upper half-space model of $\bbH^{3}$. 
Regard $\partial \bbH^{3} = \bbR^{2} \cup \{\infty\}$. 
Let $(q_{1}, \dots , q_{4},t)$ denote the element of $\mathcal{B}$ 
satisfying that $\dfrac{q_{i}-t}{\sqrt{1+t^{2}}} = \cos \alpha_{i}$ for $i = 1, \dots, 4$. 
Suppose that $p_{i} > 0$ satisfy $q_{i} = \dfrac{p_{i+1}}{p_{i}}$. 
Fix the following points in $\bbR^{2}$: 
\begin{align*}
P_{1} &= (p_{1},p_{2}), \ P_{2} = (-p_{3},p_{2}), \ P_{3} = (-p_{3},-p_{4}), \ P_{4} = (p_{1},-p_{4}), \\
R_{1} &= (p_{1},tp_{1}), \ R_{2} = (-tp_{2},p_{2}), \ R_{3} = (-p_{3},-tp_{3}), \ R_{4} = (tp_{4},-p_{4}). 
\end{align*}
Let $C_{i}$ denote the circle in $\bbR^{2}$ with the center $R_{i}$ which contains $O = (0,0)$. 
Let $S_{i}$ denote the intersectional point of $C_{i}$ and $C_{i+1}$ other than $O$. 
Let $Q_{i}$ denote the intersectional point of the lines $O S_{i}$ and $P_{i} P_{i+1}$. 
The point $Q_{i}$ may not be contained in the interior of the segment $P_{i} P_{i+1}$. 

Let $\widetilde{C}_{i}$ denote the totally geodesic plane in $\bbH^{3}$ bounded by $C_{i}$. 
Let $\widetilde{P}_{i} = \infty P_{i} \cap \widetilde{C}_{i}, \
\widetilde{Q}_{i} = \infty Q_{i} \cap \widetilde{C}_{i} \in \bbH^{3}$. 
We construct a desired holed trapezohedron $T$ as follows. 
Figure~\ref{fig:hcs-hproj.pdf} indicates the projection of $T$ to $\bbR^{2}$, 
where a neighborhood of the point $\infty$ is truncated. 
The vertices of $T$ are the points $O, \infty, \widetilde{P}_{i}$, and $\widetilde{Q}_{i}$. 
The edges of $T$ are the geodesic segments $\infty \widetilde{P}_{i}$, $O \widetilde{Q}_{i}$, $\widetilde{P}_{i} \widetilde{Q}_{i} = \hat{L}_{i}$, and the possibly holed segments between $\widetilde{Q}_{i}$ and $\widetilde{P}_{i+1}$. 
Suppose that $Q_{i}$ is not contained in the interior of the segment $P_{i}P_{i+1}$. 
In this case, the segments $O Q_{i}$ and $P_{i+1} Q_{i+1}$ intersect. 
We obtain a holed trapezohedron $T$ 
by making a hole between $\widetilde{Q}_{i}$ and $\widetilde{P}_{i+1}$ drawn in blue. 
The faces $\infty \widetilde{P}_{i} \widetilde{Q}_{i} \widetilde{P}_{i+1}$ and $O \widetilde{Q}_{i+1} \widetilde{P}_{i+1} \widetilde{Q}_{i}$ are ``holed polygons.'' 
If two points $Q_{i}$ and $Q_{i+1}$ are out of the rectangle $P_{1} P_{2} P_{3} P_{4}$, 
it is sufficient to make two holes. 
Then $T$ has the dihedral angle $\alpha_{i}$ at $\widetilde{P}_{i} \widetilde{Q}_{i} = \hat{L}_{i}$ 
and the right dihedral angle at any other edge 
in the same manner as \cite{Yoshida22}. 
\end{proof}

\fig[width=8cm]{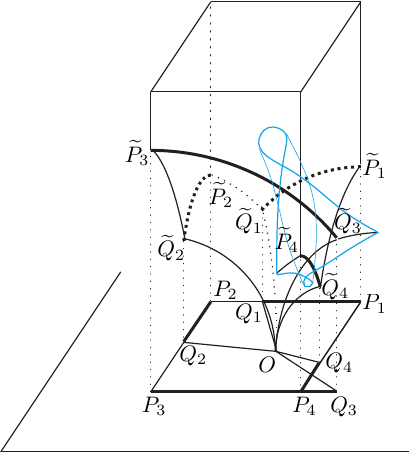}{Projection of a holed trapezohedron}

Let $\calHC_{\mathrm{sym}} (X,L)$ denote the space of holed cone structures obtained by gluing four copies of a holed trapezohedron as above. 
Let 
\[
\calC_{\mathrm{sym}} (X,L) \subset \calC (X,L) \cap \calHC_{\mathrm{sym}} (X,L)
\] 
denote the space of cone structures obtained by gluing four copies of a non-holed trapezohedron as above. 
The space $\calC_{\mathrm{sym}} (X,L)$ is the component of $\calC (X,L)$ containing the hyperbolic structure by \cite[Corollary 3.12]{Yoshida22}.
Let $\Theta \colon \calHC (X,L) \to [0, \infty)^{4}$ denote the map assigning the cone angles. 
The image $\Theta (\calHC_{\mathrm{sym}} (X,L))$ is $[0,2\pi)^{4}$ by Theorem~\ref{thm:htrap}. 
The image $\Theta (\calC_{\mathrm{sym}} (X,L)) \subset [0,2\pi)^{4}$ was explicitly described by \cite[Theorem 3.2]{Yoshida22}. 
Since a holed trapezohedron for fixed dihedral angles is unique up to holes, 
the map $\Theta \colon \calHC_{\mathrm{sym}} (X,L) \to [0,2\pi)^{4}$ is injective. 
Moreover, holed trapezohedra can be constructed to depend continuously on the dihedral angles $\alpha_{i}$. 
Hence we obtain the following corollary. 

\begin{cor}
There exists a subspace $\calHC_{\mathrm{sym}} (X,L)$ of $\calHC (X,L)$ 
containing $\calC_{\mathrm{sym}} (X,L)$ 
such that  the map $\Theta \colon \calHC_{\mathrm{sym}} (X,L) \to [0,2\pi)^{4}$ 
assigning the cone angles is a homeomorphism. 
\end{cor}

\section*{Acknowledgements} 
The author is grateful to Hirotaka Akiyoshi, Tetsuya Ito, and the anonymous reviewers for their helpful comments. 
This work was supported by the World Premier International Research Center Initiative Program, International Institute for Sustainability with Knotted Chiral Meta Matter (WPI-SKCM$^2$), MEXT, Japan, 
JSPS Program for Forming Japan's Peak Research Universities (J-PEAKS) Grant Number JPJS00420230011, 
JSPS KAKENHI Grant Number 19K14530, 
JST CREST Grant Number JPMJCR17J4, 
and the Research Institute for Mathematical Sciences, an International Joint Usage/Research Center located in Kyoto University. 

\bibliographystyle{siam}
\bibliography{ref-hcs}

\end{document}